\documentclass[11pt]{amsart}
\usepackage{amscd,amssymb}
\input xy
\xyoption{all}

\newtheorem{theorem}{Theorem}[section]
\newtheorem{lemma}[theorem]{Lemma}

\newtheorem{corollary}[theorem]{Corollary}
\newtheorem{proposition}[theorem]{Proposition}

\newtheorem{conjecture}[theorem]{Conjecture}

\renewcommand{\leq}{\leqslant}
\renewcommand{\geq}{\geqslant}

\theoremstyle{definition}

\theoremstyle{definition}
\newtheorem{remark}[theorem]{Remark}
\numberwithin{equation}{section}

\newcommand{\ve}{\varepsilon}
\newcommand{\vp}{\varphi}
\newcommand{\mn}{\sqrt{-1}}
\newcommand{\ov}[1]{\overline{#1}}
\newcommand{\de}{\partial}
\newcommand{\db}{\overline{\partial}}
\newcommand{\ddbar}{\sqrt{-1} \partial \overline{\partial}}
\newcommand{\ti}[1]{\tilde{#1}}

\newcommand{\tr}[2]{\textrm{tr}_{#1} #2}
\DeclareMathOperator{\Ric}{Ric}

\numberwithin{equation}{section} \numberwithin{figure}{section}

\author{Mark Gross}
\address{DPMMS, University of Cambridge, Wilberforce Road, Cambridge, CB3 0WB, UK}
\email{mgross@dpmms.cam.ac.uk}
\author{Valentino Tosatti}
\address{Department of Mathematics and Statistics, McGill University, Montr\'eal, Qu\'ebec H3A 0B9, Canada}
\address{Department of Mathematics, Northwestern University, 2033 Sheridan Road, Evanston, IL 60208}
\email{tosatti@math.northwestern.edu}
\author{Yuguang Zhang}
\address{Department of Mathematical Sciences, University of Bath, Bath, BA2 7AY, UK}
\address{Current address: Riemann Center for Geometry and Physics, Leibniz Universit\"at Hannover, 30167 Hannover, Germany}
\email{yuguangzhang76@yahoo.com}
\title{Geometry of twisted K\"ahler-Einstein metrics and collapsing}

\begin{document}

\begin{abstract}
We prove that the twisted K\"ahler-Einstein metrics that arise on the base of certain holomorphic fiber space with Calabi-Yau fibers have conical-type singularities along the discriminant locus. These fiber spaces arise naturally when studying the collapsing of Ricci-flat K\"ahler metrics on Calabi-Yau manifolds, and of the K\"ahler-Ricci flow on compact K\"ahler manifolds with semiample canonical bundle and intermediate Kodaira dimension. Our results allow us to understand their collapsed Gromov-Hausdorff limits when the base is smooth and the discriminant has simple normal crossings.
\end{abstract}

\maketitle

\section{Introduction}
In this paper we are concerned with the study of canonical metrics $\omega$ which exist on the base of certain holomorphic fiber space $f:M^m\to N^n$ ($n<m$) with Calabi-Yau fibers, away from the discriminant locus, and which satisfy a twisted K\"ahler-Einstein equation of the form
\begin{equation}\label{twisted}
\mathrm{Ric}(\omega)=\lambda\omega+\omega_{\rm WP},
\end{equation}
where $\lambda$ will equal $0$ or $-1$ (depending on our setup, to be described below) and $\omega_{\rm WP}$ is a semipositive definite Weil-Petersson form which encodes the variation of complex structure of the Calabi-Yau fibers of $f$.  These metrics were constructed in the work of Song-Tian \cite{ST0}, and since then metrics satisfying \eqref{twisted} have been shown to arise as collapsing limits of higher-dimensional Calabi-Yau manifolds (starting from \cite{To1} and more recently in \cite{CVZ,Fi,GTZ,GTZ2,HT,HT2,Li,Li2,OO,ST0,ST,STZ,St,TWY,TZ,TZ2,Zh}) as well as of long-time solutions of the K\"ahler-Ricci flow on compact K\"ahler manifolds with semiample canonical bundle and intermediate Kodaira dimension (starting from \cite{ST0} and more recently in \cite{EGZ,ST,STZ,TZh,To4,TWY,TZ,Zh}).

 When $M$ is holomorphic symplectic  and $f$ is a holomorphic Lagrangian fibration,
 these twisted K\"ahler-Einstein metrics coincide with the McLean-Hitchin metrics of certain semi-flat hyperk\"{a}hler structures that appear in the work of McLean \cite{McL} and Hitchin \cite{Hi,Hi2}, in which case these metrics are in fact special K\"ahler in the sense of \cite{Fr}.  For example,
   Gross-Wilson   showed in \cite{GW} that  the collapsing limits of Ricci-flat K\"ahler metrics on elliptically fibered $K3$ surfaces with $I_1$ singular fibers are McLean metrics of semi-flat hyperk\"{a}hler Lagrangian fibrations, and later Song-Tian proved  in \cite{ST0} that the limits satisfy the twisted K\"ahler-Einstein equation \eqref{twisted}.

 As mentioned earlier, the canonical metrics $\omega$ are in general singular along the discriminant locus of $f$. The nature of their singularities is intimately related to the aforementioned collapsing problems, but up to now the only situation when the singularities of $\omega$ were completely understood is when $\dim N=1$ \cite{GTZ2,He,Zh}. In this paper, we solve this problem in all dimensions by showing roughly speaking that $\omega$ has conical type singularities, and we derive consequences for these collapsing problems.

Let us now describe our results more precisely. We will work in the two above-mentioned settings separately.
\subsection{The Calabi-Yau setting}
In this setting we assume that $M^m$ is a projective Calabi-Yau manifold, that is a projective manifold with $K_M\cong\mathcal{O}_M$, and we fix $\Omega$ a never-vanishing holomorphic $m$-form on $M$. We suppose we have a fiber space $f:M\to N$ (i.e. a surjective holomorphic map with connected fibers) over a normal projective variety $N^n$. These fiber spaces arise for example when we have a semiample line bundle $L$ on $M$ and $f$ is its associated fiber space \cite[2.1.27]{Laz}. We let $D$ be the closed analytic subvariety of $N$ which is the union of the singular locus of $N$ together with the critical values of $f$ on the regular locus of $N$. We will refer somewhat informally to $D$ as the discriminant locus of $f$ and to $S=f^{-1}(D)$ as the locus of singular fibers of $f$, and by construction we have that $f:M\backslash S\to N\backslash D$ is a proper holomorphic submersion with Calabi-Yau fibers. We can then consider the fiber integral of the Calabi-Yau volume form $(-1)^{\frac{m^2}{2}}\Omega\wedge\ov{\Omega}$ via $f|_{M\backslash S}$ and obtain a smooth positive volume form $f_*((-1)^{\frac{m^2}{2}}\Omega\wedge\ov{\Omega})$ on $N_0:=N\backslash D$. We fix a K\"ahler metric $\omega_N$ on $N$ (in the sense of analytic spaces \cite{Moi} if $N$ is not smooth). In \cite{ST,To1,EGZ2,CGZ} it is shown that there is a continuous $\omega_N$-psh function $\vp$ on $N$, smooth on $N_0$, such that $\omega:=\omega_N+\ddbar\vp$ is a K\"ahler metric on $N_0$ which satisfies the twisted K\"ahler-Einstein equation \eqref{twisted} with $\lambda=0$, where $\omega_{\rm WP}\geq 0$ is the Weil-Petersson form described in \cite{ST,To1}.

An outstanding problem is to understand the behavior and singularities of the twisted K\"ahler-Einstein metric $\omega$ near the discriminant locus $D$. As explained below, this problem is intimately linked with the collapsing behavior of certain Calabi-Yau metrics on $M$. The only general results on this problem are when $N$ is a Riemann surface, in which case the precise behavior of $\omega$ was calculated in \cite[\S 3]{He} (see also \cite{GTZ2,ST0,Zh,Zh2}), and when $M$ is hyperk\"ahler \cite{TZ2}.

In general, denote by $N^{\rm reg}$ the regular locus of $N$ (whose complement has codimension at least $2$) and write $D=D^{(1)}\cup D^{(2)}$ where $D^{(1)}$ is the union of all codimension $1$ irreducible components of $D$ and $\dim D^{(2)}\leq n-2$. Let $D^{\rm snc}\subset D^{(1)}\cap N^{\rm reg}$ be the snc locus of $D^{(1)}$, so that we can write $D=D^{\rm snc}\cup D'$ where $D'$ is a closed analytic subset of codimension at least $2$ which contains $D^{(2)}$ as well as the singularities of $N$.

Write $\{D_i\}_{i=1}^\mu$ for the irreducible components of $D^{\rm snc}$ and for each $1\leq i\leq \mu$ fix a defining section $s_i$ for the line bundle corresponding to $D_{i}$ and a Hermitian metric $h_i$ on this line bundle. Since the divisor $D^{\rm snc}$ on $N^{\rm reg}\backslash D'$ has simple normal crossings support, there is a well-established notion of K\"ahler metrics with $\omega_{\rm cone}$ with conical singularities along $D^{\rm snc}$, with given cone angles $2\pi\alpha_i$ along each $D_{i}$ ($0<\alpha_i\leq 1$), as described in section \ref{sectvol}. Our first main result is then the following:

\begin{theorem}\label{prin1}
In the Calabi-Yau setting, $\omega$ extends to a smooth K\"ahler metric across $D^{(2)}\cap N^{\rm reg}$. Furthermore, there exists
a conical metric $\omega_{\rm cone}$ with conical singularities along $D^{\rm snc}$ (and cone angles in $2\pi\mathbb{Q}\cap (0,2\pi]$) such that for any $x\in D^{\rm snc}$ there is an open set $x\in U\subset N^{\rm reg}$ and constants $A,C>0$ and on $U\backslash D^{\rm snc}$ we have
\begin{equation}\label{to0bis}
C^{-1}\left(1-\sum_{i=1}^\mu\log|s_i|_{h_i}\right)^{-A\max(n-2,0)}\omega_{\rm cone}\leq\omega\leq C\left(1-\sum_{i=1}^\mu\log|s_i|_{h_i}\right)^A \omega_{\rm cone}.
\end{equation}
\end{theorem}

This is the direct generalization of the statement that is proved in \cite{He,ST0,GTZ2,Zh,Zh2} when $N$ is a Riemann surface, and it strengthens the estimate obtained in \cite{TZ2} when $M$ is hyperk\"ahler, by giving a precise quasi-isometric description of $\omega$ on near any point of $D^{\rm snc}$. The logarithmic terms that appear in \eqref{to0bis} are negligible compared to the ``poles'' of $\omega_{\rm cone}$, and this will be important in our applications to collapsing below.

In the special case when $N$ is smooth and $D^{(1)}$ has simple normal crossings, we are able to show that the better estimate
\begin{equation}\label{quasiisom}
C^{-1}\omega_{\rm cone}\leq\omega\leq C\left(1-\sum_{i=1}^\mu\log|s_i|_{h_i}\right)^A \omega_{\rm cone},
\end{equation}
holds globally on $N\backslash D^{(1)}$.

\subsection{Applications to collapsing}
In the same Calabi-Yau setting, we can fix K\"ahler metrics $\omega_M,\omega_N$ on $M$ and $N$ and $t\geq 0$, and then we let $\ti{\omega}_t=f^*\omega_N+e^{-t}\omega_M+\ddbar\vp_t$ be the unique Ricci-flat K\"ahler metric on $M$ cohomologous to $f^*\omega_N+e^{-t}\omega_M$, given by Yau's Theorem \cite{Ya}, which necessarily satisfies
\begin{equation}\label{maa}
\ti{\omega}_t^m=(f^*\omega_N+e^{-t}\omega_M+\ddbar\vp_t)^m=c_t e^{-(m-n)t}(-1)^{\frac{m^2}{2}}\Omega\wedge\ov{\Omega},\quad \sup_M\vp_t=0,
\end{equation}
where the constants $c_t$ converge to a positive limit as $t\to\infty$. A very natural question is then to understand the behavior of the Ricci-flat metrics $\ti{\omega}_t$ as $t\to\infty$, and this has been much studied in recent years, see e.g. \cite{CVZ,GTZ,GTZ2,GW,HT,HT2,Li,Li2,OO,ST0,ST,STZ,To1,TWY,TZ,TZ2} and references therein.

In particular, in \cite{To1} the second-named author proved that there is a twisted K\"ahler-Einstein metric $\omega=\omega_N+\ddbar\vp$ on $N_0=N\backslash D$ which satisfies \eqref{twisted} with $\lambda=0$, such that away from the singular fibers $S=f^{-1}(D)$ we have convergence of $\ti{\omega}_t$ to  $f^*\omega$ as $t\to\infty$ in the sense that $\vp_t\to f^*\vp$ in $C^{1,\alpha}_{\rm loc}(M\backslash S)$ for all $0<\alpha<1$. This was later improved in \cite{TWY} to locally uniform convergence of $\ti{\omega}_t$ to $f^*\omega$, and to $C^\alpha_{\rm loc}(M\backslash S)$ convergence for all $0<\alpha<1$ in \cite{HT2}. It is conjectured \cite{To2,To3} that in fact $\ti{\omega}_t\to f^*\omega$ in $C^\infty_{\rm loc}(M\backslash S)$, but this is only known in general when the generic fibers of $f$ are tori (or finite quotients) by \cite{GTZ,HT,TZ} or when $f$ is isotrivial \cite{HT2}.

In this paper we focus on the question of understanding the structure of the global Gromov-Hausdorff limit of $(M,\ti{\omega}_t)$. The following was conjectured by the second-named author \cite{To2,To3}, and is directly inspired by a similar conjecture by Gross-Wilson, Kontsevich-Soibelman and Todorov \cite{GW,KS,KS2} in the ``mirror'' setting of large complex structure limits:

\begin{conjecture}\label{con1}
Let $(X,d_X)$ denote the metric completion of $(N_0,\omega)$, and $S_X=X\backslash N_0$. Then
\begin{itemize}
\item[(a)] $(X,d_X)$ is a compact metric space and $S_X$ has real Hausdorff codimension at least $2$
\item[(b)] The Ricci-flat manifolds $(M,\ti{\omega}_t)$ converge to $(X,d_X)$ in the Gromov-Hausdorff topology as $t\to 0$
\item[(c)] $X$ is homeomorphic to $N$.
\end{itemize}
\end{conjecture}

The first result in this direction was the work of Gross-Wilson \cite{GW} for elliptically fibered $K3$ with only $I_1$ singular fibers, where $\ti{\omega}_t$ is constructed via a precise gluing construction, which in particular proves Conjecture \ref{con1}. The gluing strategy has been recently carried out for Calabi-Yau $3$-folds $M$ with $f$ a Lefschetz $K3$ fibration \cite{Li2}, and even more recently for elliptically fibered $K3$ surfaces with arbitrary singular fibers \cite{CVZ} (with some restrictions on $[\omega_M]$), but it seems completely unclear how to implement a gluing strategy in general due to the vastly different possible types of singular fibers and nontrivial geometry of the discriminant locus.

A different approach to this conjecture, initiated in \cite{To1}, is to work directly with the family of complex Monge-Amp\`ere equations \eqref{maa}, and prove suitable estimates for the solutions $\vp_t$. In this way, Conjecture \ref{con1} was completely proved when $N$ is a Riemann surface \cite{GTZ2}, or when $M$ is hyperk\"ahler \cite{TZ2}. Furthermore, it was shown in \cite{TZ2} building upon our earlier work in \cite{GTZ,GTZ2} that parts (a) and (b) of Conjecture \ref{con1} can be proved provided one can understand the blowup behavior of $\omega$ near $D$ (see section \ref{sectmain} for a precise statement), and one does not need to consider the actual Ricci-flat metrics $\ti{\omega}_t$ themselves anymore.

Very recently, in \cite{STZ} part (b) of Conjecture \ref{con1} was proved in general, and part (c) if $N$ has at worst orbifold singularities, following a different path that in particular does not give pointwise bounds for $\omega$ near $D$. As a corollary of our main theorem \ref{prin1}, which is completely independent of the results in \cite{STZ}, we are able to settle the conjecture when the codimension $1$ part of $D$ is a divisor with simple normal crossings:

\begin{corollary}\label{mainthm}
In the Calabi-Yau setting, suppose that $N$ is smooth and $D^{(1)}$ has simple normal crossings. Then Conjecture \ref{con1} holds.
\end{corollary}

In the special case when $N$ is a Riemann surface, this corollary is proved in our earlier work \cite{GTZ2}. Examples where this corollary applies with higher-dimensional base $N$ can be obtained by taking $M$ to be one of the Calabi-Yau $3$-folds constructed by Schoen \cite{Sch} as fiber products of two rational elliptic surfaces, which admit elliptic fibrations with base $N$ one of the rational elliptic surfaces, and in some cases have smooth connected discriminant locus.

In corollary \ref{mainthm} we allow the possibility that $D^{(1)}=\emptyset$, i.e. that $D$ has codimension $2$ or more. However this situation is very special, as it implies that the map $f$ is isotrivial (see Proposition \ref{extens}). In general, as stated in Theorem \ref{prin1}, $\omega$ extends smoothly across $D^{(2)}\cap N^{\rm reg}$, so all the singularities of $\omega$ are along the divisorial part $D^{(1)}$ and along $N^{\rm sing}$.

\subsection{The K\"ahler-Ricci flow setting}We now discuss the second setting, following Song-Tian \cite{ST0,ST}. In this setting $M^m$ is a compact K\"ahler manifold with $K_M$ semiample and with $0<\kappa(M)<m$ (call $n=\kappa(M)$). Using sections of $\ell K_M$ for $\ell$ sufficiently divisible, we obtain a holomorphic map $f:M\to\mathbb{P}^r$ with image $N\subset\mathbb{P}^r$ a normal projective variety of dimension $n$ (the canonical model of $M$). The map $f:M\to N$ has connected fibers, and the generic fiber is a Calabi-Yau $(m-n)$-fold. Write $\omega_N=\frac{1}{\ell}\omega_{FS}|_N$, so that $f^*\omega_N$ belongs to $c_1(K_M)$. Fix also a basis $\{s_i\}$ of $H^0(M,\ell K_M)$, which define the map $f$, and let
$$\mathcal{M}=\left((-1)^{\frac{\ell m^2}{2}}\sum_i s_i\wedge\ov{s_i}\right)^{\frac{1}{\ell}},$$
which is a smooth positive volume form on $M$ which satisfies
$$\Ric(\mathcal{M})\overset{\rm loc}{:=}-\ddbar\log\mathcal{M}=-f^*\omega_N.$$
If $D\subset N$ denotes the singular locus of $N$ together with the critical values of $f$ on the smooth part of $N$, and $S=f^{-1}(D)$, then $f:M\backslash S\to N_0:=N\backslash D$ is a proper holomorphic submersion and $f_*\mathcal{M}$ is a smooth positive volume form on $N_0$. It is known \cite{ST,To1,EGZ2,CGZ} that we can solve
$$(\omega_N+\ddbar\vp)^m = e^\vp f_*\mathcal{M},$$
with $\vp\in C^0(N)$, $\vp$ is $\omega_N$-psh and smooth on $N_0$, where $\omega:=\omega_N+\ddbar\vp$ is a K\"ahler metric that satisfies the twisted K\"ahler-Einstein equation \eqref{twisted} with $\lambda=-1$.
By \cite{ST0,ST,TWY} we know that if $\omega_0$ is any K\"ahler metric on $M$ and $\omega(t)$ is its evolution by the K\"ahler-Ricci flow
$$\frac{\de}{\de t}\omega(t)=-\Ric(\omega(t))-\omega(t),\quad\omega(0)=\omega_0,$$
then $\omega(t)$ exists for all $t\geq 0$ and as $t\to\infty$ we have that $\omega(t)\to f^*\omega$ in $C^0_{\rm loc}(M\backslash S)$.

One than has the completely analogous conjecture to Conjecture \ref{con1} for the collapsed Gromov-Hausdorff limit of $(M,\omega(t))$, but even less is known about it in general, with the only case that has been settled so far is when $\kappa(M)=1$ and the general fiber of $f$ is a torus (or finite quotient), by \cite{STZ, TZh}.

Our result in this setting is again that the twisted K\"ahler-Einstein metric $\omega$ on $N_0$ has conical-type singularities along $D^{\rm snc}$:
\begin{theorem}\label{prin2}
In the K\"ahler-Ricci flow setting, the same results as in Theorem \ref{prin1} hold for $\omega$, namely $\omega$ extends to a smooth K\"ahler metric across $D^{(2)}\cap N^{\rm reg}$, it satisfies \eqref{to0bis} near every point of $D^{\rm snc}$, and when $N$ is smooth and $D^{(1)}$ has simple normal crossings the estimate \eqref{quasiisom} holds globally.
\end{theorem}

In \cite{STZ,TZh}, the analog of \eqref{quasiisom}, proved in \cite{He,GTZ2,Zh} was used as a key ingredient to understand the collapsed Gromov-Hausdorff limit of the flow when $\dim N=1$ and one assumes a locally uniform Ricci curvature bound away from the singular fibers of $f$. It follows from \cite{FZ,GTZ,HT,TZ} (see the exposition in \cite[Theorem 5.24]{To4}) that the assumption on the Ricci curvature of $\omega(t)$ always holds the the smooth fibers of $f$ are tori or finite quotients of tori, and it is expected to hold in general \cite[Conjecture 4.7]{Ti2}.

\subsection{Strategy of proof}
The strategy of proof is the following. First, we pass to a resolution of singularities $\pi:\ti{N}\to N$ such that $\ti{N}$ is smooth and $\pi^{-1}(D)=E$ is a divisor with simple normal crossings, and $\pi$ is an isomorphism over the locus where $N$ is regular and $D^{(1)}$ is a simple normal crossings divisor. In particular, we may choose $\pi=\mathrm{Id}$ in the case when $N$ is smooth and $D^{(1)}$ has simple normal crossings. Then we prove in section \ref{sectvol} that the volume form of $\pi^*\omega$ satisfies the estimate
\begin{equation}\label{to22}
C^{-1}H\omega_{\rm cone}^n\leq (\pi^*\omega)^n\leq CH\left(1-\sum_{i=1}^\mu\log|s_i|_{h_i}\right)^A \omega_{\rm cone}^n,
\end{equation}
where $H$ is a smooth nonnegative function which is bounded above and may go to zero near the exceptional components of $E$. Estimate \eqref{to22} is obtained using Hodge theory, like in \cite{GTZ2} where a weaker estimate was obtained when $\dim N=1$. Very recently, similar estimates were obtained in more general settings in \cite{Ki,Ta2} using different ideas.

The main task is then to deduce a pointwise bound for $\omega$ of the form \eqref{to0bis} from \eqref{to22}. First, in section \ref{codim2} we show that $\omega$ extends smoothly across the analytic subset $D^{(2)}\cap N^{\rm reg}$, which has codimension at least $2$, so we it suffices to understand the behavior of $\omega$ near the divisor $D^{(1)}$.

Next, if $N$ is smooth and $D^{(1)}$ has simple normal crossings (so that $\pi=\mathrm{Id}$), we show in section \ref{sectsnc} that \eqref{quasiisom} holds, using estimates for Monge-Amp\`ere equations with conical singularities as developed in \cite{Br,CGP,DS,GP,JMR}.

In the general case, the resolution $\pi$ will be nontrivial, which introduces the major issue that the class $[\pi^*\omega_N]$ where the form $\pi^*\omega$ lives is not K\"ahler anymore, and nothing is known about ``conical'' metrics in such classes. In section \ref{sectgen}, we show that on $\ti{N}\backslash E$ we have the ``Tsuji-type'' \cite{Ts} estimate
\begin{equation}\label{to01}
\pi^*\omega\leq \frac{C}{|s_F|^{2A}}\left(1-\sum_{i=1}^\mu\log|s_i|_{h_i}\right)^A \omega_{\rm cone},
\end{equation}
where $F\subset E$ denotes the exceptional components of $E$. This, together with \eqref{to22}, easily implies the estimate \eqref{to0bis} since $\pi$ is an isomorphism where $N$ is smooth and $D^{(1)}$ is snc. This proves the main Theorem \ref{prin1}, and the application to collapsing of Ricci-flat metrics in Corollary \ref{mainthm} then follows from this and our earlier work \cite{TZ2} (extending ideas that we introduced in \cite{GTZ2} when $\dim N=1$), as explained in section \ref{sectmain}. Indeed there we showed that to prove parts (a) and (b) of Conjecture \ref{con1} in general it suffices to show the estimate
\begin{equation}\label{to00}
\pi^*\omega\leq C\left(1-\sum_{i=1}^\mu\log|s_i|_{h_i}\right)^A \omega_{\rm cone},
\end{equation}
and that part (c) also follows from this provided that $\pi=\mathrm{Id}$. Of course, when $\pi=\mathrm{Id}$ the estimates \eqref{to01} and \eqref{to00} are identical, but it remains an open problem to show that \eqref{to00} holds in general when $\pi$ is nontrivial (which would prove part (a) of Conjecture \ref{con1}).

Lastly, in section \ref{sectkrf} we extend these arguments to the K\"ahler-Ricci flow setting, proving Theorem \ref{prin2}. The arguments are very similar to the Calabi-Yau setting, with the only main difference being that now $K_M$ is replaced by $\ell K_M$ (which instead of being trivial is semiample). We show via a standard ramified covering trick that estimates analogous to \eqref{to22} also hold in this setting, by reducing to the case when $\ell=1$ which is treated in section \ref{sectvol}.\\

{\bf Acknowledgments. }We are grateful to H. Guenancia, H.-J. Hein, D. Kim, Y. Li, M. Popa  and P.M.H. Wilson for discussions. This work was done during the second-named author's visits to the Institut Henri Poincar\'e in Paris in 2018 (supported by a Chaire Poincar\'e at IHP funded by the Clay Mathematics Institute) and to the Center for Mathematical Sciences and Applications at Harvard University in 2018, and during the third-named author's visit to Northwestern University in 2016, which we would like to thank for the hospitality and support. The first-named author was supported by EPSRC grant EP/N03189X/1 and a Royal Society Wolfson Research Merit Award. The second-named author was also partially supported by NSF grants DMS-1610278 and DMS-1903147. The third-named author was supported by the Simons Foundation's program Simons Collaboration on Special Holonomy in Geometry, Analysis and Physics (grant \#488620).

\section{Estimates on the volume form}\label{sectvol}

In this section we use Hodge theory to derive asymptotics for the pushforward of the holomorphic volume form on a Calabi-Yau manifold which is the total space of a fiber space over a lower-dimensional base. There is by now a large literature on related questions, including for example \cite{Ba,Berm,BJ,CKS,CKS2,CTs,EFM,EFM2,Fu,GTZ,Ka,Ki,Ko,MT1,MT2,MT3,RZ,Sc,Ta,Ta2,TZ2,Wa,Wa2,Yo,Zh}, with the very recent papers \cite{Ki,Ta2} explicitly focusing on the case when $\dim N\geq 2$. Here we decided to follow closely our previous work \cite[Section 2]{GTZ2} where we proved the upper bound in \eqref{estimate1} in the case when $\dim N=1$. Our arguments generalize this to bases of arbitrary dimensions, and also provide the lower bound in \eqref{estimate1} which crucially matches the upper bound up to logarithmic error terms. We note that a more refined analysis can in principle be carried out as in \cite{CKS,CKS2,Ko} to obtain precisely matching upper and lower bounds, with a more complicated logarithmic behavior. However, since for our purposes the cruder estimate \eqref{estimate1} suffices, we have decided not to pursue this.

The setup in this section will be the Calabi-Yau setting described in the introduction, namely we let $M$ be an $m$-dimensional projective Calabi-Yau manifold ($K_M\cong\mathcal{O}_M$) with holomorphic $m$-form $\Omega$, $N$ a normal projective $n$-dimensional variety, $0<n<m$, and assume that we have a surjective holomorphic map $f:M\rightarrow N$ with connected fibers, and let $D\subset N$ be the discriminant locus of the map $f$, i.e. the union of the singular locus of $N$ together with the critical values of $f$ on $N^{\rm reg}$. As in the introduction, we write $D=D^{(1)}\cup D^{(2)}$ so that $D^{(1)}$ is a divisor and $D^{(2)}$ has codimension at last $2$.

By Hironaka's resolution of singularities, there is a birational morphism $\pi:\tilde N\rightarrow N$  such that $\tilde N$ is nonsingular and $E=\pi^{-1}(D)$ is a simple normal crossings divisor, and such that $\pi:\tilde N \backslash E\to N\backslash D$ is an isomorphism. In fact, it is known (see e.g. \cite{Te}) that we can choose $\pi$ so that it is an isomorphism over the locus where $N$ is regular and $D^{(1)}$ has simple normal crossings. This locus is Zariski open in $N$, with complement of codimension $2$.

Let $\phi:\tilde M\to M\times_N \tilde N$ be a birational morphism, inducing a map $\tilde f:\tilde M\to\tilde N$ such that
\begin{enumerate}
\item $\tilde M$  is nonsingular
\item $\tilde f^{-1}(E)$ is simple normal crossings in $\tilde M$
\item $E$ is the discriminant locus of $\tilde f$.
\end{enumerate}
Again, this can be accomplished by Hironaka's theorem.

From these, it follows by elementary arguments (see e.g. \cite[Setup 2.3, Lemma 2.5]{MT3}) that there is a Zariski open subset $\tilde N^o\subset \tilde N$ with $\tilde N\backslash E\subsetneq \tilde N^o\subset \tilde N\backslash Sing(E)$, such that $\tilde f^{-1}(\tilde N^o)\to \tilde N^o$ is a normal crossings morphism, i.e., locally one can find coordinates $(z_1,\ldots,z_m)$ on $\tilde f^{-1}(\tilde N^o)$ and coordinates $(y_1,\dots,y_n)$ on $\tilde N^o$ such that $\tilde f$ is given by
$$(z_1,\ldots,z_m)\mapsto \left(\prod_{i=1}^{m-n+1}z_i^{d_i},z_{m-n+2},\dots,z_m\right).$$
Note that the fibres may be non-reduced, and $d_i$ are multiplicities of irreducible components of $\ti{f}^{-1}(E)$.

Given a point $y_0\in E$ let $U$ be an open neighborhood of $y_0$ in $\tilde N$ with $U\cong \Delta^n$ (the unit polydisc) and $U\backslash E\cong(\Delta^*)^\ell\times \Delta^{n-\ell}$, for some $1\leq \ell\leq n$, and with local holomorphic coordinates $(y_1,\dots,y_n)$ on $U$ (centered at $y_0$). If we write $\Omega$ also for the pullback of $\Omega$ to $\tilde M$, then we can define a real nonnegative function $\vp_U$ on $U\backslash D$ by
\[
  \tilde{f}_*(\Omega\wedge\bar\Omega)= \varphi_U dy_1\wedge\dots\wedge dy_n\wedge d\bar y_1\wedge\dots d\bar y_n.
\]
The main result of this section is then:

\begin{theorem}\label{volest}
After possibly shrinking $U$, we have
\begin{equation}
\label{estimate1}
C^{-1}\prod_{i=1}^\ell|y_i|^{-2\left(1-\alpha_i\right)} \leq\varphi_U(y)\le C\prod_{i=1}^\ell|y_i|^{-2\left(1-\alpha_i\right)} \left(1-\sum_{i=1}^\ell\log |y_i|\right)^d,
\end{equation}
for some non-negative integer $d$ and rational numbers $\alpha_i\in\mathbb{Q}_{>0}$, and for all $y\in U\backslash E$.
\end{theorem}
Note that we do not assert that necessarily $\alpha_i\leq 1$, so some of the exponents $-2\left(1-\alpha_i\right)$ may be positive.
\begin{remark}
As is well-known (cf. \cite{BJ,Zh}) from the proof it follows that when $\dim N=1$ the rational number $\alpha$ equals the log canonical threshold of the fiber $f^{-1}(0)$, an algebro-geometric invariant. See \cite{Ki} for a similar interpretation in the general case.
\end{remark}
\begin{proof}
Let $n_0\in U\backslash E$ be a basepoint and let
\[
T_i:H^{m-n}(\tilde f^{-1}(n_0),\mathbb{C})\rightarrow H^{m-n}(\tilde
f^{-1}(n_0),\mathbb{C})
\]
be the monodromy operator for a loop based at $n_0$ around the $i^{th}$ copy of $\Delta^*$, $1\leq i\leq \ell$.
By the Monodromy Theorem (see e.g., the appendix of \cite{La}, or \cite[Chap.II, Application 17]{Gr}) $T_i$ is quasi-unipotent with $(T_i^{m_i}-I)^{d_i}=0$ for some positive integers $d_i$ and $m_i$, and $m_i$ is the least common multiple of the multiplicities of the irreducible components over a general point of $\{y_i=0\}$. Set also $m_i=1$ for $\ell<i\leq n$. Let $\bar U=\Delta^n$ with coordinate $(w_1,\dots,w_n)$, and let $\mu:\bar U\rightarrow U$ be given by $\mu(w_1,\dots,w_n)=(w_1^{m_1},\dots,w_n^{m_n})$.

Pull-back and normalize the family $\tilde M\rightarrow \tilde N$ via the composition $\bar U {\smash{ \mathop{\longrightarrow}\limits^{\mu}}} U\hookrightarrow \tilde N$,
to obtain a family $\bar f:\bar M\rightarrow \bar U$.
This has discriminant locus $\bar D=\mu^{-1}(U\cap E)$ and $\bar f$ now has the property that the monodromy around a loop in $\bar U^o:=\bar U\setminus \bar D\cong (\Delta^*)^\ell\times \Delta^{n-\ell}$ is unipotent.

Now the trivial vector bundle $\mathcal{ H}^{m-n}=(R^{m-n}\bar f_*\mathbb{C})\otimes_{\mathbb{C}}\mathcal{O}_{\bar U^o}$ on $\bar U^o$ comes with the Gauss-Manin connection, whose flat sections are sections of $R^{m-n}\bar f_*\mathbb{C}$. It is standard that this vector bundle has a canonical extension to $\bar U$, (see e.g., \cite{Gr}, Chapter IV or \cite{Sc})
constructed as follows.
Choosing a basepoint $t_0\in \bar U^o$, let $e_1,\ldots,e_s$ be a basis for $H^{m-n}(\bar f^{-1}(t_0))$. These extend to multi-valued flat sections of $\mathcal{H}^{m-n}$, which we write as $e_i(w)$. However,
\[
\sigma_i(w):=\exp\left(-\sum_{j=1}^\ell N_j\frac{\log w_j}{2\pi\sqrt{-1}}
\right)e_i(w)
\]
with $N_j=\log T_j$ is in fact a single-valued holomorphic section of $\mathcal{H}^{m-n}$. We then extend $\mathcal{H}^{m-n}$ across $\bar U$ by decreeing these sections to form a holomorphic frame for the vector bundle. Call this extension $\mathcal{H}^{m-n}_{\bar U}$.

It is then standard (see again \cite{Gr}, Chapter IV) that the Hodge bundle $F^{m-n}_{\bar U^o}:=(f_*\Omega^{m-n}_{\bar M/\bar U})|_{\bar U^o}\subset
\mathcal{H}^{m-n}$ has a natural extension $F^{m-n}_{\bar U}\subset \mathcal{H}^{m-n}_{\bar U}$ to $\bar U$.

Next note that the form
\[
\Omega^{rel}:=\iota(\partial/\partial y_1,\dots,\de/\de y_n)\Omega
\]
is a well-defined section of $\tilde f_*\Omega^{m-n}_{\tilde f^{-1}(U)/U}$ and thus pulls back to a well-defined section $\Omega^{rel}_{\bar U^o}$ of $F^{m-n}_{\bar U^o}$. Furthermore,
the function $\varphi_U$ given in the statement of the theorem satisfies at a point $y\in U$
\[
\varphi_U(y)=(-1)^{\frac{(m-n)^{2}}{2}}\int_{\tilde f^{-1}(y)} \Omega^{rel}\wedge \bar\Omega^{rel}.
\]

We will show that the section $\Omega^{rel}_{\bar U^o}$ of $F^{m-n}_{\bar U^o}$ extends to a meromorphic section of $F^{m-n}_{\bar U}$ and investigate the order of the poles of this section along the discriminant locus. For this it is sufficient to restrict to one-parameter families. Let $S_i\subset U$ be the disc on which all the coordinates $y_j$ are constant and
nonzero except for $y_i\in \Delta$. Let $\bar S_i=\mu^{-1}(S_i)$. Note that $\bar S_i$ is a disjoint union of discs. Denote by $\tilde M_i\to S_i$ and $\bar M_i\to \bar S_i$ the basechanges of $\tilde M\to \tilde N$ and $\bar M\to\bar U$ respectively. By choosing $S_i$ generally, we can assume that $\tilde M_i\to S_i$ is a normal crossings morphism, possibly with a nonreduced singular fiber over $0\in S_i$. Call this fiber $\tilde Y$. In particular, restricting to $S_i$, we can consider $\Omega^{rel}$ as a meromorphic section of  $\tilde f_*\Omega^{m-n}_{\tilde M_i/S_i}(\log \tilde Y)$.

Going back to our problem, we first determine the order of pole of $\Omega^{rel}$ at $0$ as a section of this bundle.
Locally on $\tilde M_i$, near a general point of an irreducible component $Y_j$  of $\tilde Y$, the map $\tilde{M}_i\to S_i$ is given by $y_i=z_1^\beta$, with $\beta\ge 1$ and $z_1,\ldots,z_{m-n+1}$ coordinates on $\tilde M_i$. We can write
$$\Omega_i:=\iota(\de_{y_1},\dots,\widehat{\de_{y_i}},\dots,\de_{y_n})\Omega,$$
as a form on $\tilde M_i$ locally as
\[
\Omega_i=\psi_i dz_1\wedge\cdots\wedge dz_{m-n+1},
\]
for some holomorphic function $\psi_i$.   We have $\psi_i=z_1^{k}\phi$ where $|\phi|>0$ at the generic point of $Y_j$.

In our local coordinate description, the vector field on $\tilde M_i$ given by $\beta^{-1}z_1^{-\beta+1} \partial_{z_1}$ is a lift of $\partial_{y_i}$. Thus $\Omega^{rel}$ as a section
of $\Omega^{m-n}_{\tilde M_i /S_i}(\log\tilde Y)$ is locally given by
\[
\pm\frac{\psi_i}{\beta z_1^{\beta-1}}dz_2\wedge \cdots \wedge dz_{m-n+1}=\phi z_1^{k+1-\beta}dz_2\wedge \cdots \wedge dz_{m-n+1}.
\]

We now need to pull-back $\Omega^{rel}$ to $\Omega^{rel}_{\bar U^o}$ and study this section as a section of $F^{m-n}_{\bar U}$.
The stable reduction theorem \cite{KKMS} gives a resolution of singularities $\bar M'_i\rightarrow \bar M_i$ such that the composed map $\bar M'_i\rightarrow \bar S_i$ is normal crossings. So we have a diagram
\[
\xymatrix@C=30pt
{ \bar M'_i\ar[r]^{\pi'}\ar[rd]_{\bar f'} & \bar M_i\ar[d]_{\bar f}
\ar[r]^{\pi}&\tilde M_i\ar[d]^{\tilde f}\\
&\bar S_i\ar[r]_{\mu}&S_i}
\]
Furthermore, the map $\pi'$ is a toric resolution of singularities by the construction of \cite{KKMS}, Chapter II. In particular, locally $\pi'$ and $\pi$ can be described as dominant morphisms of toric varieties of the same dimension. On such toric charts, by \cite{Oda}, Prop.\ 3.1, the sheaves of logarithmic differentials $\Omega^{m-n}_{\tilde M_i/S_i}(\log \tilde Y),
\Omega^{m-n}_{\bar M_i/\bar S_i}(\log \bar Y)$, and $\Omega^{m-n}_{\bar M'_i/\bar S_i}(\log \bar Y')$ are trivial vector bundles generated by exterior products of logarithmic differentials of toric monomials (here $\bar Y$ and $\bar Y'$ denote the fibers over zero), and thus $\pi^*\Omega^{m-n}_{\tilde M_i/S_i}(\log \tilde Y)\cong \Omega^{m-n}_{\bar M_i/\bar S_i}(\log \bar Y)$
and $(\pi')^*\Omega^{m-n}_{\bar M_i/\bar S_i}(\log \bar Y) \cong \Omega^{m-n}_{\bar M'_i/\bar S_i}(\log \bar Y')$.
Furthermore,
\begin{align*}
\bar f'_*\Omega^{m-n}_{\bar M'_i/\bar S_i}(\log \bar Y')
\cong {} &
\bar f_*\pi'_*\Omega^{m-n}_{\bar M'_i/\bar S_i}(\log \bar Y')\\
\cong {} & \bar f_*\pi'_*(\pi')^*\Omega^{m-n}_{\bar M_i/\bar S_i}(\log\bar Y)\\
\cong {} & \bar f_*((\pi'_*\mathcal{O}_{\bar M'_i})\otimes
\Omega^{m-n}_{\bar M_i/\bar S_i}(\log\bar Y))\\
\cong {} & \bar f_*\Omega^{m-n}_{\bar M_i/\bar S_i}(\log \bar Y).
\end{align*}
It also follows from \cite{Steen} (see also \cite{Gr}, Chapter VII) that
\[
F^{m-n}_{\bar U}|_{S_i}
\cong \bar f'_*\Omega^{m-n}_{\bar M'_i/\bar S_i}(\log \bar Y').
\]
Thus, in order to understand the behaviour of $\Omega^{rel}_{\bar U^o}$ as a section of $F^{m-n}_{\bar U}$, it is sufficient to pull back $\Omega^{rel}$ to $\Omega^{m-n}_{\bar M_i/\bar S_i}(\log \bar Y)$ and understand the behaviour of this form as a section of $\bar f_*\Omega^{m-n}_{\bar M_i/\bar S_i} (\log \bar Y)$.

Again, we do this locally near the inverse image of a general point of an irreducible component of $\tilde Y$. Using the same notation as before, we know that $\bar M_i$ is locally given by the normalization of the equation $w_i^{m_i}=z_1^\beta$. Note that $\beta|m_i$, so a local description of the normalization is given by an equation $w_i^{m_i/\beta}=\xi z_1$ for $\xi$ an $\beta$-th root of unity. Thus $\Omega^{rel}$ pulls back to
\[
\psi_i w_i^{\frac{-m_i(\beta-1)}{\beta}}dz_2\wedge\cdots \wedge dz_{m-n+1}= \phi w_i^{\frac{-m_i(\beta-1-k)}{\beta}}dz_2\wedge\cdots \wedge dz_{m-n+1}.
\]

Note that $$w_i^{\frac{-m_i(\beta-1-k)}{\beta}}=y_i^{\frac{1+k}{\beta}-1}.$$
Thus letting $\beta_i$ and $k_i$ be the numbers such that $\frac{1+k_i}{\beta_i}= \min \{\frac{1+k}{\beta}\}$ among irreducible
components of $\tilde Y$ (and let $Y_j$ be one of the components where this minimum is achieved), we find $w_i^{m_i(\beta_i-1-k_i)/\beta_i}\Omega^{rel}_{\bar U^o}$
extends to a holomorphic section of $\Omega^{m-n}_{\bar M_i/\bar S_i}(\log \bar Y)$, hence yields a holomorphic section of $F^{m-n}_{\bar U}|_{\bar S_i}=
\bar f'_*\Omega^{m-n}_{\bar M'_i/\bar S_i}(\log \bar Y')$.

Now set
\[
\Omega^{norm}:= \prod_{i=1}^\ell w_i^{m_i(\beta_i-1-k_i)/\beta_i} \Omega^{rel}_{\bar U^o}.
\]
This now extends to a holomorphic section of $F^{m-n}_{\bar U}$.  By construction, since $|\phi|>0$ at the generic point of $Y_j$, we see that $\Omega^{norm}\neq 0$ at the generic point of $Y_j$. Thus we can write $\Omega^{norm}$, as a section of $\mathcal{H}^{m-n}_{\bar U}$, as
\[
\Omega^{norm}=\sum_{i=1}^s h_i(w)\sigma_i(w),
\]
for $h_i$ holomorphic functions on $\bar U$. We then compute, with $\langle\cdot,\cdot\rangle$ denoting the cup product followed by evaluation on the fundamental class
\[
H^{m-n}(f^{-1}(t_0),\mathbb{C})\times H^{m-n}(f^{-1}(t_0),\mathbb{C})\rightarrow
\mathbb{C},
\]
that
\begin{align*}
&\int_{\bar f^{-1}(w)}\Omega^{norm}\wedge \bar\Omega^{norm}\\
= {} & \left\langle \sum_{i=1}^s h_i(w)\sigma_i(w), \sum_{j=1}^s \bar h_j(w)
\bar \sigma_j(w)\right\rangle\\
= {} &
\left\langle
\sum_{i=1}^se^{-\sum_{k=1}^\ell N_k\log w_k/2\pi\sqrt{-1}} h_ie_i,
\sum_{j=1}^se^{ \sum_{p=1}^\ell N_p\log \bar w_p/2\pi\sqrt{-1}}
\bar h_je_j\right\rangle.
\end{align*}
Note the exponentials can be expanded in a finite power series because the $N_i$ are nilpotent, and hence a term in the above expression is
\[
C\cdot h_i\bar h_j \prod_{k=1}^\ell(\log w_k)^{d_k}(\log \bar w_k)^{d'_k}
\left\langle \prod_{k} N_k^{d_k}e_i, \prod_k N_k^{d'_k}e_j\right\rangle.
\]
Here the constant $C$ only depends on the powers $d_k$, $d'_k$ occuring.
We can assume that we have chosen the imaginary part of $\log w_k$ to lie between $0$ and $2\pi$ (this is equivalent to choosing the branch of $e_i(y)$). Keeping in mind that the $h_j$ are holomorphic on $\bar U$, after shrinking $\bar U$ we can assume that $|h_j|$ are bounded by some constant, and so we see that the above term is bounded by a sum of a finite
number of expressions of the form
\[
C'\prod_{k=1}^\ell (-\log |w_k|)^{d''_k}.
\]
Thus the entire integral is bounded by an expression of the form
\[
C\left(1- \sum_{k=1}^\ell \log |w_k|\right)^d,
\]
for suitable choice of constant $C$ and exponent $d$.

Returning to $\Omega^{rel}_{\bar U^o}$, we see that
\[
(-1)^{\frac{(m-n)^{2}}{2}}\int_{\bar f^{-1}(w)}\Omega^{rel}_{\bar U^o}\wedge \bar\Omega^{rel}_{\bar U^o}
\le C\prod_{i=1}^\ell|w_i|^{-2m_i(\beta_i-1-k_i)/\beta_i}\left((1-\sum_{k=1}^\ell\log |w_k|\right)^d.
\]
Using $y_i=w_i^{m_i}$ then gives one of the inequalities in \eqref{estimate1}, with $\alpha_i=\frac{1+k_i}{\beta_i}$.

To prove the other inequality, we want to  show that $$ (-1)^{\frac{(m-n)^{2}}{2}}\int_{ f^{-1}(y)}\Omega^{rel}\wedge \bar\Omega^{rel}\geq C^{-1}\prod_{i=1}^{\ell}  |y_{i}|^{-2(1-\frac{k_i+1}{\beta_{i}})}$$   for $y\in U\backslash\tilde{D}=(\Delta^{*})^{\ell}\times \Delta^{n- \ell}\subset \ti{N}\backslash E$.
Let $y_{1}=w_{1}^{m_{1}}, \cdots, y_{\ell}=w_{\ell}^{m_{\ell}}, y_{\ell+1}=w_{\ell+1}, \cdots, y_{n}=w_{n}$ be the covering $\mu:\bar{U}^{o} \rightarrow U^{o} $ such that the monodromy operators are unipotent, and let $$\Omega^{norm}=\prod_{i=1}^{\ell} w_{i}^{m_{i}(\beta_i-1-k_i)/\beta_i}\Omega^{rel}_{\bar U^o}. $$  Then our goal is equivalent to proving
 \begin{equation}\label{goalll}
 (-1)^{\frac{(m-n)^{2}}{2}}\int_{\bar f^{-1}(w)}\Omega^{norm}\wedge \bar\Omega^{norm}\geq C^{-1}>0.
 \end{equation}

 Let $Z$ be the zero divisor of $\Omega^{norm}$, which is regarded as a section of  the  line bundle $F^{m-n}_{\bar U}$.  Then $Z\subset \bar{D}=\{w_{1}\cdots w_{\ell}=0\}$ the discriminant locus. To prove \eqref{goalll} it is enough to show that $Z=\emptyset$.

 Assume that $Z\neq \emptyset$.  Note that $Z$ is a divisor, and thus $Z$ contains an  irreducible component of $\bar{D}$, say $\{w_{i}=0\}$.  Let $S_{i}$ be a disc on which all $y_{j}$ are constant nonzero except for $y_{i}$, i.e.  $S_{i}=\{y_{1}=c_{1}, \cdots, \widehat{y_{i}=c_{i}},\cdots,y_{n}=c_{n} \}$,   and $\bar{S}_{i}=\mu^{-1}(S_{i})$. Hence $w_{i}$ is the coordinate on $\bar{S}_{i}$, and  $\bar{S}_{i}\bigcap Z=\{w_{i}=0\}\neq \emptyset$.  Let $\tilde{M}_{i}\rightarrow S_{i}$ and $\bar{M}_{i}\rightarrow \bar{S}_{i}$ be the base changes. From above,  $\tilde{M}_{i}\rightarrow S_{i}$ is normal crossing morphism with a possible nonreduced central fiber $\tilde{Y}$.

As before, let $Y_j$ be one of the components of $\ti{Y}$ where the $\min \{\frac{1+k}{\beta}\}$ is achieved (notation as before).
  Then in local coordinates on $\bar{M}_i$ as before, near the preimage of $Y_j$ in $\bar{M}_i$ under the map  $q:\bar M_i\to M_i$, we have that $\Omega^{rel}_{\bar U^o}$ pulls back to $$\phi  w_i^{\frac{-m_i(\beta-1-k)}{\beta}}dz_2\wedge\cdots \wedge dz_{m-n+1},$$
   where $\phi$ is not identically zero, and so $\Omega^{norm}$ pulls back to $$\phi dz_2\wedge\cdots \wedge dz_{m-n+1},$$
  with $\phi\not\equiv 0$, which is a contradiction because the preimage $q^{-1}(\ti{Y})$ under the map $q:\bar M_i\to M_i$ is mapped into $Z=\{\Omega^{norm}=0\}$ by $\bar{M}_i\to \bar S_i\hookrightarrow\bar{U}$. Hence $Z=\emptyset$.
\end{proof}

Next, we translate the local estimate in Theorem \ref{volest} into a global statement.
To do this, we need to recall the notion of K\"ahler metrics with conical singularities. We write $E=\bigcup_{j=1}^\mu E_j$ for the decomposition of $E$ in irreducible components.
Given real numbers $0<\alpha_i\leq 1$, $1\leq i\leq \mu$, there is a well-defined notion of {\em conical K\"ahler metrics} $\omega_{\rm cone}$ on $\ti{N}$ with singularities along $E$ with cone angle $2\pi\alpha_i$ along each component $E_i$. Any such metric is a smooth K\"ahler metric on $\ti{N}\backslash E$ such that in any local chart $U$ (a unit polydisc with coordinates $(w_1,\dots, w_n)$) centered at a point of $E$ adapted to the normal crossings structure (so $E\cap U$ is given by $w_1\cdots w_k=0$ for some $1\leq k\leq n$, and say that $\{w_i=0\}=E_{j_i}\cap U$ for some $1\leq j_i\leq \mu$ and all $1\leq i\leq k$) we have that on $U\backslash\{w_1\cdots w_k=0\}$ the metric $\omega_{\rm cone}$ is uniformly equivalent to the model
$$\sum_{i=1}^k \frac{\mn dw_i\wedge d\ov{w}_i}{|w_i|^{2(1-\alpha_{j_i})}}+\sum_{i=k+1}^n \mn dw_i\wedge d\ov{w}_i.$$
We also fix a defining section $s_i$ of the divisor $E_i$ and a smooth Hermitian metric $h_i$ on $\mathcal{O}(E_i)$, for all $1\leq i\leq \mu$.
An explicit conical K\"ahler metric $\omega_{\rm cone}$ on $N_0$ with cone angle $2\pi\alpha_i$ along each component $E_i$ is given by
$$\omega_{\rm cone}=\omega_{\ti{N}}+\delta\ddbar\left(\sum_{i=1}^\mu |s_i|_{h_i}^{2\alpha_i}\right),$$
for some small $\delta>0$, where $\omega_{\ti{N}}$ is any given K\"ahler metric on $\ti{N}$.

The following is then the global version of Theorem \ref{volest}, which was stated as \eqref{to22} in the Introduction:
\begin{theorem}\label{main2}
There is a constant $C>0$ and natural numbers $d\in\mathbb{N}$, $0\leq p\leq \mu$ and rational numbers $\beta_i>0,$ $1\leq i\leq p$ and $0<\alpha_i\leq 1$, $p+1\leq i\leq \mu$, such that on $\pi^{-1}(N_0)$ we have
\begin{equation}\label{to2}
C^{-1}\prod_{j=1}^p |s_j|_{h_j}^{2\beta_j}\omega_{\rm cone}^n\leq (\pi^*\omega)^n\leq C\prod_{j=1}^p|s_j|_{h_j}^{2\beta_j}\left(1-\sum_{i=1}^\mu\log|s_i|_{h_i}\right)^d \omega_{\rm cone}^n,
\end{equation}
where $\omega_{\rm cone}$ is a conical metric with cone angle $2\pi\alpha_i$ along the components $E_i$ with $p+1\leq i\leq \mu$.
\end{theorem}
\begin{proof}
This makes use of the above results from Hodge theory.
First of all, as explained for example in \cite[Section 4]{To1}, on $N_0$ the metric $\omega$ satisfies \begin{equation}\label{eq3} \omega^{n}=cf_{*}((-1)^{\frac{m^{2}}{2}}\Omega \wedge \overline{\Omega}) \end{equation}
for some explicit constant $c>0$. Recall that at the beginning of this Section we have constructed a sequence of blowups $\pi:\ti{N}\to N$ such that $E=\pi^{-1}(D)$ is an snc divisor, and then we have the base-change $M\times_N \ti{N}\to M$ and its resolution $\ti{M}\to M\times_N \ti{N}$ which, when composed, gives a holomorphic map $p:\ti{M}\to M$. We also have the new fibration $\ti{f}:\ti{M}\to\ti{N}$ so that we have the commutative diagram
\[
\xymatrix@C=30pt
{  \ti{M}\ar[d]_{\ti{f}}
\ar[r]^{p}&M\ar[d]^{f}\\
\ti{N}\ar[r]_{\pi}&N}
\]
For clarity, call $\ti{\Omega}=p^*\Omega$, which is a holomorphic $n$-form on $\ti{M}$ which in general has zeros. Then on $\ti{N}\backslash E$ we actually have equality
\begin{equation}\label{eq4}
\pi^*f_*(\Omega\wedge\ov{\Omega})=\ti{f}_*(\ti{\Omega}\wedge\ov{\ti{\Omega}}),
\end{equation}
since $p,\pi$ are biholomorphisms when we are away from the singular fibers. Combining \eqref{eq3} and \eqref{eq4} we get
$$(\pi^*\omega)^n=c\ti{f}_*((-1)^{\frac{m^{2}}{2}}\ti{\Omega}\wedge\ov{\ti{\Omega}})$$
on $\ti{N}\backslash E$. Now Theorem \ref{volest} gives us an estimate for the behavior of the RHS locally, when compared with a Euclidean volume form, near any given point of $E$, of the form
\[\begin{split}
C^{-1}\prod_{i=1}^\ell|y_i|^{-2\left(1-\alpha_i\right)} &\leq\frac{c\ti{f}_*((-1)^{\frac{m^{2}}{2}}\ti{\Omega}\wedge\ov{\ti{\Omega}})}{idy_1\wedge d\ov{y}_1\wedge\dots\wedge idy_n\wedge d\ov{y}_n}\\
&\le C\prod_{i=1}^\ell|y_i|^{-2\left(1-\alpha_i\right)} \left(1-\sum_{i=1}^\ell\log |y_i|\right)^d,
\end{split}\]
where $\alpha_i\in\mathbb{Q}_{>0}$. We keep those $\alpha_i$ which satisfy $\alpha_i\leq 1$, and for the others we define $\beta_j=-1+\alpha_j\in\mathbb{Q}_{>0}$ so that we can write (with abuse of notation)
$$\prod_{k=1}^\ell|y_k|^{-2\left(1-\alpha_k\right)}=\prod_{i}|y_i|^{-2\left(1-\alpha_i\right)}\prod_{j}|y_j|^{2\beta_j},$$
and so that this shows that \eqref{to2} holds locally,
in the sense that given any point $x\in E$ we can find a neighborhood $U$ of $x$ in $\ti{N}$ together with $\ell,C,d,p,\alpha_i,\beta_i$ (which all depend on $U$), such that
\begin{equation}\label{estimate1a}
\frac{C^{-1}\prod_{j}|y_j|^{2\beta_j}\omega_{\ti{N}}^n}{\prod_{i}|y_i|^{2\left(1-\alpha_i\right)}}\leq (\pi^*\omega)^n\leq \left(1-\sum_{i=1}^\ell\log|y_i|\right)^d \frac{C\prod_{j}|y_j|^{2\beta_j}\omega_{\ti{N}}^n}{\prod_{i}|y_i|^{2\left(1-\alpha_i\right)}},
\end{equation}
holds on $U$. Fix attention on one irreducible component $E_1$ of $E$, and cover $E_1$ by such open sets $U_j$, and say that on each $U_j$ we have $E_1\cap U_j=\{y_1=0\}$. We wish to show that the number $\alpha_1$ (which give the cone angle along $E_1$ if $\alpha_1\leq 1$) or $\beta_1$ (which give the order of zero along $E_1$ if $\alpha_1>1$) that we get is the same an all these $U_j$'s. On a nonempty intersection of two such $U_j$'s, we can look at points approaching $E_1$ in this intersection, which do not come close to any of the other components of $E$. Then the blowup rate (or order of vanishing) of the function in \eqref{estimate1a} along $\{y_1=0\}$ must be the same in both sets, which thanks to \eqref{estimate1a} implies that the corresponding numbers $\alpha_1$ (or $\beta_1$) are equal. Repeating this for all components of $E$, this proves \eqref{to2}.
\end{proof}

\section{The codimension $2$ part of the discriminant}\label{codim2}
Still working in the Calabi-Yau setting, our next goal is to show that the codimension $2$ (or more) subset $D^{(2)}\cap N_0$ does not contribute to the singularities of $\omega$ on $N_0$. For this, we fix a K\"ahler metric $\omega_N$ on $N$ (in the sense of analytic spaces \cite{Moi} if $N$ is not smooth), and we define a smooth positive function $\mathcal{F}$ defined on $N_0$  by
\begin{equation}\label{deff}\frac{(-1)^{\frac{m^2}{2}}f_*(\Omega\wedge\ov{\Omega})}{\omega_N^n}=\mathcal{F},
\end{equation}
which satisfies
$$\ddbar\log \mathcal{F}=-\omega_{\rm WP}+\Ric(\omega_N),$$
while the metric $\omega$ on $N_0$ satisfies
\begin{equation}\label{asino}
\omega^n=c\mathcal{F}\omega_N^n,
\end{equation}
for some constant $c>0$.

Our next observation is that the components of $D$ of codimension $2$ or higher do not contribute to the singularities of $\mathcal{F}$. More precisely, write $N'=N^{\rm reg}\backslash D^{(1)}$ and $D'=D^{(2)}\cap N'$, so $D'$ is codimension $2$ or more inside the complex manifold $N'$, and $N'\backslash D'=N_0$.

\begin{proposition}\label{cod2}
The function $\mathcal{F}$ on $N_0$ extends to a smooth positive function on all of $N'$, and the current $\omega$ on $N$ is a smooth K\"ahler metric on $N'$.
\end{proposition}
\begin{proof}
For this we use the well-known relation between the Weil-Petersson form and the Hodge metric on the Hodge bundle, that we now recall. The reader is referred to e.g. \cite{Gr} for the relevant background. For $y\in N_0$ denote by $M_y=f^{-1}(y)$ the fiber over $y$ and let $P\subset H^{n-m}(M_y,\mathbb{C})=:H$ be the primitive cohomology given by the K\"ahler class $[\omega_M|_{M_y}]$, $H_\mathbb{Z}:=P\cap H^{n-m}(M_y,\mathbb{Z})$, $k^{p,q}:=\dim_{\mathbb{C}} P\cap H^q(M_y,\Omega^p_{M_y}) (p+q=n-m)$, and let $Q$ be the polarization form on $H$. Let $\mathcal{D}$ be the classifying space for integral polarized Hodge structures of type $\{H_{\mathbb{Z}}, k^{p,q}, Q\}$, and
let
$$\mathcal{P}:N_0\to\mathcal{D}/\Gamma,$$
be the corresponding period map, where $\Gamma$ is the monodromy group. This map is holomorphic, locally liftable and horizontal.
Since $\mathcal{D}$ has a Hermitian metric with negative holomorphic sectional curvature in the horizontal directions \cite{GS}, a result of Griffiths-Schmid \cite[Corollary 9.8]{GS} shows that $\mathcal{P}$ admits a holomorphic extension $\ov{\mathcal{P}}:N'\to \mathcal{D}/\Gamma$ across the analytic set $D'$ of codimension $2$ or more.

In our situation, the form $\omega_{\rm WP}$ on $N_0$ is in fact equal to the pullback $\mathcal{P}^*\omega_{\rm H}$ of the Hodge K\"ahler metric on $\mathcal{D}/\Gamma$ by \cite{Ti}. It follows therefore that $\omega_{\rm WP}$ admits a smooth extension $\ov{\omega}_{\rm WP}$ to $N'$, which is equal to $\ov{\mathcal{P}}^*\omega_{\rm H}\geq 0$.

Given any $x\in D'$, there is an open subset $x\in U\subset N'$, $U$ biholomorphic to a ball, where we can write
$$-\ov{\omega}_{\rm WP}+\Ric(\omega_N)=\ddbar u,$$ for some smooth function $u$. Note that $U\cap N_0$ is simply connected. On the other hand, if on $N_0$ we write $v=\log \mathcal{F}$, then the difference $w=u-v$ is pluriharmonic on $U\cap N_0$. In particular, $\de w$ is $d$-closed on $U\cap N_0$, and since $H^1(U\cap N_0,\mathbb{C})=0$, it is $d$-exact, i.e. $\de w=dw'$ for some complex function $w'$ on $U\cap N_0$. Then $\db w'=0$ so $w'$ is holomorphic, and also $d(w'+\ov{w'}-w)=0$ so $w'+\ov{w'}-w=c$ on $U\cap N_0$. Therefore $w$ equals the real part of the holomorphic function $h=2w'-c$ on $U\cap N_0$. By Hartogs, $h$ extends to a holomorphic function on $U$, and so $w$ (and hence also $v$) extends smoothly on $U$. This shows the desired extension property of $\mathcal{F}$.

Lastly, the fact that $\omega$ is a smooth K\"ahler metric on $N'$ follows from the way in which $\omega$ (solving \eqref{asino}) is constructed in \cite[\S 3.2]{ST} via an approximation procedure, after we know that $\mathcal{F}$ is smooth and positive on $N'$.
\end{proof}

It is interesting to point out the following extension of the results in \cite{TZ0}:
\begin{proposition}\label{extens}
Let $M$ be a projective manifold with $K_M\cong\mathcal{O}_M$, $N$ a compact K\"ahler manifold and $f:M\to N$ a surjective holomorphic map with connected fibers. If the locus $D\subset N$ of critical values of $f$ has codimension $2$ or more, then $f|_{M\backslash f^{-1}(D)}:M\backslash f^{-1}(D)\to N\backslash D$ must be a holomorphic fiber bundle (i.e. $f$ is isotrivial).
\end{proposition}
The result in \cite{TZ0} (see also \cite[Theorem 3.3]{TZ2}) is the same theorem under the stronger assumption that $D=\emptyset$. On the other hand, this result is of course false if $D$ is allowed to have components of codimension $1$.
\begin{proof}
Since $f$ has connected fibers and $N$ is smooth we have $f_*\mathcal{O}_M\cong\mathcal{O}_N$. Combining this with $K_M\cong\mathcal{O}_M$ and with the projection formula we get that
$$\mathcal{O}_N\cong f_*(K_{M/N})+ K_N,$$
i.e.
$$-K_N\cong f_*(K_{M/N}).$$
In particular the pushforward $f_*(K_{M/N})$ is indeed a line bundle on $N$. On $N\backslash D$ we have a smooth Hermitian metric on this line bundle given locally by
$$(-1)^{\frac{(m-n)^2}{2}}\int_{M_y}\Omega_y\wedge\ov{\Omega_y},$$
(where $M_y=f^{-1}(y)$)
whose curvature form on $N\backslash D$ is exactly the Weil-Petersson form $\omega_{\rm WP}\geq 0$. From the proof of Proposition \ref{cod2} we know that $\omega_{\rm WP}$ admits a smooth extension $\ov{\omega}_{\rm WP}$ to $N$.
On the other hand, there is another smooth metric on $-K_N$ whose curvature form on $N$ is $\Ric(\omega_N)$, and so there exists a smooth function $u$ on $N\backslash D$ such that
$$\Ric(\omega_N)+\ddbar u =\omega_{\rm WP}.$$
Since $D\subset N$ has codimension $2$ or more, the Grauert-Remmert extension theorem \cite{GR} shows that $u$ extends to a quasi-psh function $\underline{u}$ on $N$ which satisfies
$\Ric(\omega_N)+\ddbar \underline{u}\geq 0$ weakly as currents. The difference $\ov{\omega}_{\rm WP}-(\Ric(\omega_N)+\ddbar \underline{u})$ is a closed real $(1,1)$ current which is the difference of two positive currents, and which is supported on the analytic set $D$ of codimension at least $2$, hence by the Federer support theorem this current vanishes.
It follows that $\ddbar\underline{u}$ is smooth on all of $N$, and hence so is $\underline{u}$ by regularity of the Laplacian. This means that $\Ric(\omega_N)+\ddbar \underline{u}\geq 0$ is a smooth semipositive form on $N$ cohomologous to $c_1(-K_N)$, and so by Yau's Theorem \cite{Ya} the manifold $N$ admits a K\"ahler metric with nonnegative Ricci curvature.

As in \cite{TZ0}, we can apply the Schwarz Lemma \cite[Lemma 3.3]{TZ0} to the map $\ov{\mathcal{P}}:N\to \mathcal{D}/\Gamma$, to conclude that $\ov{\mathcal{P}}$ is constant on $N$. Therefore $\mathcal{P}$ is constant on $N\backslash D$, and as in \cite[Proof of Theorem 3.1]{TZ0} this implies that the map $f|_{M\backslash f^{-1}(D)}$ is a holomorphic fiber bundle.
\end{proof}

\section{The case when the discriminant locus is an snc divisor}\label{sectsnc}
In this section we prove the estimate \eqref{quasiisom}, under the assumption that $N$ is smooth and $D^{(1)}$ is a simple normal crossings divisor.

We start with a preliminary observation related to the volume estimate in Theorem \ref{main2}, in the general setting of Theorem \ref{prin1}.
Recall that we have constructed a resolution $\pi:\ti{N}\to N$ of $N$ such that $\pi^{-1}(D)=E$ is a simple normal crossings divisor, $E=\cup_{i=1}^\mu E_i$, and we have fixed defining sections $s_i$ and metrics $h_i$ for $\mathcal{O}(E_i)$.
Let $\ti{D}\subset \ti{N}$ be the proper transform of $D^{(1)}$, and write
$$E=F\cup\ti{D},$$
where $F$ is a $\pi$-exceptional divisor.

We fix a K\"ahler metric $\omega_{\ti{N}}$ on $\ti{N}$ and a K\"ahler metric $\omega_N$ on $N$ (in the sense of analytic spaces \cite{Moi} if $N$ is not smooth), and define a smooth nonnegative function $J_\pi$ on $\ti{N}$ by
\begin{equation}\label{jac}
\pi^*\omega_N^n = J_\pi \omega_{\ti{N}}^n,
\end{equation}
which is clearly positive on $\pi^{-1}(N_0)$.

Going back to the volume form estimate \eqref{to2}, we wish to understand which of the components of $F$ or $\ti{D}$ actually have zeros (i.e. $\beta_j>0$) or conical singularities (i.e. $\alpha_i<1$). The first  observation is the following:

\begin{lemma}\label{moron}
The function $\mathcal{F}$ on $N_0$ defined in \eqref{deff} satisfies
\begin{equation}\label{lower}
\mathcal{F}\geq C^{-1}.
\end{equation}
\end{lemma}
\begin{proof}
This is a consequence of results of Song-Tian \cite{ST}, as follows. First, as shown in \cite[Lemma 3.3]{ST}, on $M\backslash S$ we can write
\begin{equation}\label{srf}
f^*\mathcal{F}=\frac{(-1)^{\frac{m^2}{2}}\Omega\wedge\ov{\Omega}}{f^*\omega_N^n\wedge\omega_{\rm SRF}^{m-n}},
\end{equation}
where $\omega_{\rm SRF}$ is a {\em semi-Ricci-flat form} on $M\backslash S$ defined by $\omega_{\rm SRF}=\omega_M+\ddbar\rho$, where $\omega_M$ is any given K\"ahler metric on $M$ with fiber volume $\int_{M_y}\omega_M^{m-n}=1$ for all $y\in N_0$, and $\rho$ is the unique smooth function on $M\backslash S$ such that
$$\omega_{\rm SRF}|_{M_y}=\omega_M|_{M_y}+\ddbar\rho|_{M_y},$$
is the unique Ricci-flat K\"ahler metric on $M_y$ cohomologous to $\omega_M|_{M_y}$ (for all $y\in N_0$), with the normalization that
$\int_{M_y}\rho|_{M_y}\omega_M|_{M_y}^{m-n}=0,$ which is obtained by a fiberwise application of Yau's Theorem \cite{Ya}.

We can then argue in a similar way as the proof of \cite[Proposition 3.2]{ST}. Explicitly, given any $y\in N_0$, on the fiber $M_y$ the volume forms $(\omega_M|_{M_y})^{m-n}$ and $(\omega_{\rm SRF}|_{M_y})^{m-n}$ have the same total volume, and hence there is a point $x_y\in M_y$ such that
$$\frac{(\omega_M|_{M_y})^{m-n}}{(\omega_{\rm SRF}|_{M_y})^{m-n}}(x_y)=1.$$
Then by \eqref{srf} we have
$$\mathcal{F}(y)=\frac{(-1)^{\frac{m^2}{2}}\Omega\wedge\ov{\Omega}}{f^*\omega_N^n\wedge\omega_{\rm SRF}^{m-n}}(x_y)=\frac{(-1)^{\frac{m^2}{2}}\Omega\wedge\ov{\Omega}}{f^*\omega_N^n\wedge\omega_{M}^{m-n}}(x_y)\geq C^{-1},$$
since on $M$ we have $f^*\omega_N^n\wedge\omega_{M}^{m-n}\leq C\omega_M^n$ and $(-1)^{\frac{m^2}{2}}\Omega\wedge\ov{\Omega}\geq C^{-1}\omega_M^n$.
\end{proof}
Pulling back \eqref{asino} via $\pi$, we can use Lemma \ref{moron} to get on $\pi^{-1}(N_0)$
$$(\pi^*\omega)^n=c(\mathcal{F}\circ\pi) \pi^*\omega_N^n=c(\mathcal{F}\circ\pi) J_\pi \omega_{\ti{N}}^n\geq C^{-1}J_\pi \omega_{\ti{N}}^n,$$
while by \eqref{to2}
$$(\pi^*\omega)^n\leq C\prod_{j=1}^p|s_j|_{h_j}^{2\beta_j}\left(1-\sum_{i=1}^\mu\log|s_i|_{h_i}\right)^d \frac{ \omega_{\ti{N}}^n}{\prod_{i=p+1}^\mu |s_i|_{h_i}^{2(1-\alpha_i)}},$$
and so
\begin{equation}\label{asino2}
C^{-1}J_\pi\leq \left(1-\sum_{i=1}^\mu\log|s_i|_{h_i}\right)^d \frac{\prod_{j=1}^p|s_j|_{h_j}^{2\beta_j} }{\prod_{i=p+1}^\mu |s_i|_{h_i}^{2(1-\alpha_i)}}.
\end{equation}
Given a component of $\ti{D}$, given by $\{s_i=0\}$ for some $1\leq i\leq \mu$, near its generic point we have that $J_\pi>0$ so by \eqref{asino2} we must have $i>p$ (i.e. there is no zero for $(\pi^*\omega)^n$ near this point).

On the other hand, for components of $F$, $(\pi^*\omega)^n$ may have either a zero or a conical type singularity, but if it has a zero, it cannot vanish more than $J_\pi$ (up to much smaller logarithmic terms), thanks again to \eqref{asino2}.

In particular, if $N$ is smooth and $D^{(1)}$ has simple normal crossings, then by definition we can take $\ti{N}=N, \pi=\mathrm{Id}$ and $E=\ti{D}=D$, so this discussion shows that in this case all the $\alpha_i$ that appear in Theorem \ref{volest} must be $\leq 1$, or equivalently that there are no $\beta_j$'s in Theorem \ref{main2} (i.e. $p=0$). In other words, this shows that:

\begin{corollary}\label{main3}
Assume that $N$ is smooth and $D^{(1)}=\cup_{j=1}^\mu D_j$ has simple normal crossings.
Then there are constants $C,d>0$ and rational numbers $0<\alpha_j\leq 1$, $1\leq j\leq \mu$, such that on $N_0$ we have
\begin{equation}\label{sncest}
C^{-1}\omega_{\rm cone}^n\leq \omega^n\leq C\left(1-\sum_{i=1}^\mu\log|s_i|_{h_i}\right)^d \omega_{\rm cone}^n,
\end{equation}
where $\omega_{\rm cone}$ is a conical metric with cone angle $2\pi\alpha_j$ along the component $D_j,1\leq j\leq\mu.$
\end{corollary}

We can now give the proof of \eqref{quasiisom}.
\begin{proof}[Proof of \eqref{quasiisom}]
On $N$ we have the continuous function $\vp$ (normalized with $\sup_N\vp=0$), which satisfies
$$\omega=\omega_N+\ddbar\vp,$$
and solves the Monge-Amp\`ere equation
\begin{equation}\label{ma0}\omega^n=(\omega_N+\ddbar\vp)^n=c (-1)^{\frac{m^2}{2}}f_*(\Omega\wedge\bar{\Omega})=\frac{\psi}{\prod_{j=1}^\mu |s_j|_{h_j}^{2(1-\alpha_j)}}\omega_N^n,
\end{equation}
on $N_0$ (and also globally on $N$ in the weak sense), where the smooth function $\psi$ on $N_0$ is defined by this last equality. In particular,
\begin{equation}\label{psh}\begin{split}
\ddbar\log(1/\psi)&=\Ric(\omega)-\Ric(\omega_N)-\sum_{j=1}^\mu (1-\alpha_j)\ddbar\log|s_j|_{h_j}^2\\
&\geq -\Ric(\omega_N)+\sum_{j=1}^\mu (1-\alpha_j)R_{h_j}\\
&\geq -C\omega_N,
\end{split}\end{equation}
holds on $N_0$.
Thanks to Corollary \ref{main3}, on $N_0$ we have that
\begin{equation}\label{dream}
C^{-1}\leq \psi\leq C\left(1-\sum_{i=1}^\mu\log|s_i|_{h_i}\right)^d.
\end{equation}
As shown in \eqref{psh} function $\log(1/\psi)$ is quasi-psh, and thanks to \eqref{dream} it is bounded above near $N\backslash N_0$, and so it extends uniquely to a quasi-psh function on $N$ (still denoted by $\log(1/\psi)$) which satisfies the same inequality \eqref{psh} by a classical result of Grauert-Remmert \cite{GR}.

Thanks to Demailly's regularization theorem \cite{De}, we can find smooth functions $u_j$ on $N$ which decrease pointwise to $\log(1/\psi)$ as $j\to\infty$ and which satisfy $\ddbar u_j\geq -C\omega_N$ and $u_j\leq C$ on $N$.

We now employ a trick from Datar-Song \cite{DS}.
For each $i=1,\dots,\mu$, let $\omega_{{\rm cone}, i}$ be a conical K\"ahler metric with cone angle $2\pi\alpha_i$ along the component $D_i$, and smooth everywhere else. For example, we can take
$$\omega_{{\rm cone},i}=\omega_N+\delta\ddbar |s_i|_{h_i}^{2\alpha_i},$$
and let also $\eta_i=\delta|s_i|_{h_i}^{2\alpha_i}.$ Then for all $j\geq 1$ let
$$G_{j,i}=\prod_{k\neq i}\left(|s_k|^2_{h_k}+\frac{1}{j}\right)^{1-\alpha_k},$$
which are smooth functions on $N$ which satisfy
\[\begin{split}
\ddbar\log G_{j,i}&=\sum_{k\neq i}(1-\alpha_k)\ddbar\log\left(|s_k|^2_{h_k}+\frac{1}{j}\right)\\
&\geq \sum_{k\neq i}(1-\alpha_k)\frac{|s_k|^2_{h_k}}{|s_k|^2_{h_k}+\frac{1}{j}}\ddbar\log |s_k|^2_{h_k}\\
&\geq -C\omega_N,
\end{split}\]
everywhere on $N$, for a constant independent of $j$.

We can now consider the complex Monge-Amp\`ere equation
\begin{equation}\label{ma}
(\omega_N+\ddbar\vp_{j,i})^n = c_{j,i} e^{-u_j}\frac{\omega_N^n}{G_{j,i} |s_i|_{h_i}^{2(1-\alpha_i)}},
\end{equation}
where $c_{j,i}$ are constants that are obtained by integrating the equation, so $c_{j,i}\to 1$ as $j\to\infty$, the functions $\vp_{j,i}$ are normalized by $\sup_N\vp_{j,i}=0$, and $\omega_{j,i}:=\omega_N+\ddbar\vp_{j,i}$ are now conical K\"ahler metrics with cone angle $2\pi\alpha_i$ along $D_i$ (only), and smooth away from $D_i$. This last fact follows from the main result of \cite{JMR} or \cite{GP} (see also \cite{Br,CGP} for earlier weaker results and \cite{Al} for more recent work). It is quickly verified that the ratio $\omega_{j,i}^n/\omega_N^n$ has $L^p$ norm uniformly bounded independent of $j$, for some $p>1$, and so by Ko\l odziej \cite{Kol} we have $$\sup_N |\vp_{j,i}|\leq C,$$
and since the right hand side of \eqref{ma} converges to the right hand side of \eqref{ma0} (say in $L^1(N)$), Ko\l odziej's stability theorem \cite{Kol2} gives us that $\vp_{j,i}\to \vp$ in $C^0(N)$ as $j\to\infty$.

Our goal is to prove that
\begin{equation}\label{goalo}
\omega_{j,i}\geq C^{-1}\omega_{{\rm cone},i},
\end{equation}
holds on $N_0$, for a constant $C$ independent of $j$. If this is proved, then it follows from this and \eqref{ma} (by standard arguments) that on any compact subset $K\subset N_0$ we have uniform $C^k$ bounds for $\omega_{j,i}$ independent of $j$ (but depending on $K$), and since $\vp_{j,i}\to \vp$ uniformly, we conclude that $\omega_{j,i}\to \omega$ in $C^\infty_{\rm loc}(N_0)$ as $j\to\infty$ (for all $i$). Therefore
$$\omega\geq C^{-1}\omega_{{\rm cone},i},$$
holds on $N_0$, and summing these up for $i=1,\dots,\mu$ we obtain
$$\omega\geq C^{-1}\omega_{{\rm cone}},$$
on $N_0$, which is half of \eqref{to0bis}.
To prove the other half of \eqref{to0bis} it suffices to use \eqref{ma0} and \eqref{dream} to obtain
$$\tr{\omega_{{\rm cone}}}{\omega}\leq (\tr{\omega}{\omega_{{\rm cone}}})^{n-1}\frac{\omega^n}{\omega_{{\rm cone}}^n}\leq C \frac{\omega^n}{\omega_{{\rm cone}}^n}\leq C\psi\leq C\left(1-\sum_{i=1}^\mu\log|s_i|_{h_i}\right)^d,$$
on $N_0$, which thus proves \eqref{to0bis}.

So it remains to prove \eqref{goalo}. The argument is similar to the one in \cite{DS}. First of all, by a calculation in the appendix of \cite{JMR} (or the simpler proof in \cite[Lemma 3.14]{Yan} due to J. Sturm), the bisectional curvature of $\omega_{{\rm cone},i}$ on $N\backslash D_i$ has a uniform upper bound. Second of all, the Ricci curvature of $\omega_{j,i}$ on $N\backslash D_i$ satisfies
$$\Ric(\omega_{j,i})=\ddbar u_j+\Ric(\omega_N)+\ddbar\log G_{j,i}+(1-\alpha_i)\ddbar\log |s_i|^2_{h_i}\geq-C\omega_{N},$$
independent of $j$. Combining these two into Yau's Schwarz Lemma calculation \cite{Ya2} gives us
$$\Delta_{\omega_{j,i}}\log\tr{\omega_{j,i}}{\omega_{{\rm cone},i}}\geq -C\tr{\omega_{j,i}}{\omega_{{\rm cone},i}}-C,$$
on $N\backslash D_i$. To ensure that the maximum is achieved away from $D_i$, we compute
$$\Delta_{\omega_{j,i}}\log(|s_i|_{h_i}^{2\gamma}\tr{\omega_{j,i}}{\omega_{{\rm cone},i}})\geq -C\tr{\omega_{j,i}}{\omega_{{\rm cone},i}}-C,$$
where we used that $\omega_N\leq C \omega_{{\rm cone},i}$ on $N\backslash D_i$, and the constant $C$ does not depend on $\gamma>0$ small. On the other hand
$$\Delta_{\omega_{j,i}}(\vp_{j,i}-\eta_i)=n-\tr{\omega_{j,i}}{\omega_{{\rm cone},i}},$$
so we get the differential inequality
$$\Delta_{\omega_{j,i}}\log(|s_i|_{h_i}^{2\gamma}\tr{\omega_{j,i}}{\omega_{{\rm cone},i}}-A(\vp_{j,i}-\eta_i))\geq \tr{\omega_{j,i}}{\omega_{{\rm cone},i}}-C,$$
for $A$ large and uniform (independent of $j$). The maximum principle can then be applied because $\omega_{j,i}$ and $\omega_{{\rm cone},i}$ are conical metrics with the same cone angles along $D_i$ hence the function $\tr{\omega_{j,i}}{\omega_{{\rm cone},i}}$ is uniformly bounded above on $N\backslash D_i$, so the maximum is achieved away from $D_i$. Using the uniform $C^0$ bound for $\vp_{j,i}$, we conclude that
$$\tr{\omega_{j,i}}{\omega_{{\rm cone},i}}\leq \frac{C}{|s_i|_{h_i}^{2\gamma}},$$
holds on $N\backslash D_i$ with constant independent of $\gamma$ and $j$, so we can let $\gamma\to 0$ and get \eqref{goalo}.
\end{proof}

\section{The main theorem}\label{sectgen}
In this section we prove the main theorem \ref{prin1}, in the Calabi-Yau setting.

Let $\pi:\ti{N}\to N$ be the composition of blowups constructed in section \ref{sectvol}, with $\ti{N}$ smooth, $\pi^{-1}(D)=E$ a simple normal crossings divisor, and $\pi$ an isomorphism on the locus where $D^{(1)}$ is snc. We write $E=\cup_{i=1}^\mu E_i$, and we have fixed defining sections $s_i$ and metrics $h_i$ for $\mathcal{O}(E_i)$.

If we denote by $F=\cup_i F_i\subset E$ the exceptional components of $E$, then it is well-known (see e.g. \cite[Lemma 7]{PS}) that for all $0<\delta\ll 1$ the form
\begin{equation}\label{asin}
\pi^*\omega_N-\delta \sum_i R_{F_i},
\end{equation}
is a K\"ahler metric on $\ti{N}$, where $R_{F_i}$ is the curvature of a certain Hermitian metric on $\mathcal{O}(F_i)$.

Next, from section \ref{sectvol} we obtain a conical structure $\{E_i,2\pi \alpha_i\}$ on $\ti{N}$, and we fix a conical K\"ahler metric $\omega_{\rm cone}$ with these cone angles around the $E_i$'s. We will choose it of the form $\omega_{\rm cone}=\omega_{\ti{N}}+\ddbar\eta,$ where $$\eta=C^{-1}\sum_i |s_i|^{2\alpha_i}_{h_i},$$ for some $C>0$ sufficiently large.
Note that of course we have $\omega_{\ti{N}}\leq C \omega_{\rm cone}$.
We also obtain the order of zeros $\beta_j$ (along some components of $F$, different from the components that have conical singularities), and we define
$$H=\prod_j |s_j|_{h_j}^{2\beta_j},$$
where the product is only over those components $F_j$ of $F$ which have a nontrivial order of zero $\beta_j>0$. On $\ti{N}\backslash F$ we have
\begin{equation}\label{ein}
\ddbar\log H=-\sum_j \beta_jR_{F_j}.
\end{equation}

We can then define a smooth function
$\psi$ on $\ti{N}\backslash E$ by
\begin{equation}\label{ma4}
\psi=\frac{\pi^*\omega^n\prod_i |s_i|^{2(1-\alpha_i)}}{H\omega_{\ti{N}}^n},
\end{equation}
which satisfies
\begin{equation}\label{sat}
C^{-1}\leq \psi\leq C\left(1-\sum_{i=1}^\mu\log|s_i|_{h_i}\right)^d,
\end{equation}
by Theorem \ref{main2}. We thus have the Monge-Amp\`ere equation
\begin{equation}\label{maeqn}
(\pi^*\omega_N+\ddbar\pi^*\vp)^n=\psi H\frac{\omega_{\ti{N}}^n}{\prod_i |s_i|^{2(1-\alpha_i)}}.
\end{equation}

The first step is to regularize equation \eqref{maeqn} as follows.
On $\ti{N}\backslash E$ we have
\begin{equation}\label{low}\begin{split}
\ddbar\log(1/\psi)&=\Ric(\pi^*\omega)-\Ric(\omega_{\ti{N}})+\sum_i (1-\alpha_i)R_i+\ddbar\log H\\
&\geq -\Ric(\omega_{\ti{N}})+\sum_i (1-\alpha_i)R_i-\sum_j \beta_jR_{F_j},
\end{split}\end{equation}
and note that the RHS is a smooth form on all of $\ti{N}$.
Also $\log(1/\psi)$ is bounded above near $E$ by \eqref{sat}, so by Grauert-Remmert \cite{GR} it extends to a global quasi-psh function on $\ti{N}$ which satisfies \eqref{low} in the weak sense. Thanks to \eqref{sat}, the extension has vanishing Lelong numbers, so we can approximate it using Demailly's regularization theorem \cite{De} by a decreasing sequence of smooth functions $u_j$ which satisfy
\begin{equation}\label{zwei}
\ddbar u_j\geq -\Ric(\omega_{\ti{N}})+\sum_i (1-\alpha_i)R_i-\sum_j \beta_jR_{F_j}-\frac{1}{j}\omega_{\ti{N}},
\end{equation}
on all of $\ti{N}$.
We also have that $\ddbar \log H\geq -C\omega_{\ti{N}}$ weakly on $\ti{N}$ and so $\log H$ can also be regularized by smooth functions $v_k$ with $\ddbar v_k\geq -C\omega_{\ti{N}}$ and $v_k\leq C$ on $\ti{N}$.

Let $\omega_{j,k}=\pi^*\omega_N+\frac{1}{j}\omega_{\ti{N}}+\ddbar\vp_{j,k}$ be a solution on $\ti{N}$ of
\begin{equation}\label{ma1}
\omega_{j,k}^n=c_{j,k} e^{v_k-u_j}\frac{\omega_{\ti{N}}^n}{\prod_i |s_i|^{2(1-\alpha_i)}},
\end{equation}
with $\omega_{j,k}$  a cone metric with the same cone angles as $\omega_{\rm cone}$. This exists thanks to the main result of \cite{JMR} or \cite{GP} (see also \cite{Br,CGP} for earlier weaker results and \cite{Al} for more recent work). In particular, $\omega_{j,k}$ is smooth on $\ti{N}\backslash E$.

It is quickly verified that the ratio $\omega_{j,k}^n/\omega_{\ti{N}}^n$ has $L^p$ norm uniformly bounded independent of $j$, for some $p>1$, and so by \cite{EGZ2} we have $$\sup_{\ti{N}} |\vp_{j,k}|\leq C,$$
and since $e^{v_k-u_j}\to H\psi$, we conclude using the stability theorem in \cite{EGZ2} that as $j,k\to\infty$ we have $c_{j,k}\to 1$ and $\vp_{j,k}\to\pi^*\vp$ in $C^0(\ti{N})$, where $\pi^*\vp$ solves \eqref{maeqn}.

We also introduce a partial regularization, which we denote by $\omega_{j}=\pi^*\omega_N+\frac{1}{j}\omega_{\ti{N}}+\ddbar\vp_{j}$, which are K\"ahler metrics on $\ti{N}\backslash E$ which solve
\begin{equation}\label{ma2}
\omega_{j}^n=c_{j} H e^{-u_j}\frac{\omega_{\ti{N}}^n}{\prod_i |s_i|^{2(1-\alpha_i)}}.
\end{equation}
Again we have $c_j\to 1$, and $c_{j,k}\to c_j$ for each $j$ fixed as $k\to\infty$. The solution $\vp_j\in C^0(\ti{N})$ exists by \cite{ST,To1,EGZ2,CGZ}, and are smooth away from $E$ by Yau \cite{Ya}. Again we have that $\vp_{j,k}\to\vp_j$ in $C^0(\ti{N})$ for each $j$ fixed as $k\to\infty$.

\begin{proposition}\label{p1}
For each $j$ there is a constant $C_j$ which depends on $j$ such that on $\ti{N}\backslash E$ we have
\begin{equation}\label{go2}
\tr{\omega_{\rm cone}}{\omega_{j,k}}\leq C_j,
\end{equation}
for all $k$.
\end{proposition}
This, together with \eqref{ma1} and standard arguments, implies that $\omega_{j,k}\to\omega_j$ locally smoothly away from $E$, and so on $\ti{N}\backslash E$ we also have
\begin{equation}\label{go1}
\tr{\omega_{\rm cone}}{\omega_{j}}\leq C_j.
\end{equation}
\begin{proof}
This follows from the estimates of Guenancia-P\u{a}un \cite{GP}. Specifically, as in \cite[proof of Theorem A]{GP}, we make one further approximation of \eqref{ma1} by a smooth PDE
\begin{equation}\label{ma3}
\omega_{j,k,\ve}^n=c_{j,k,\ve} e^{v_k-u_j}\frac{\omega_{\ti{N}}^n}{\prod_i (|s_i|^2+\ve^2)^{(1-\alpha_i)}},
\end{equation}
where $\omega_{j,k,\ve}=\pi^*\omega_N+\frac{1}{j}\omega_{\ti{N}}+\ddbar\vp_{j,k,\ve}$ is now a smooth K\"ahler metric on $\ti{N}$. Again we have $\vp_{j,k,\ve}\to\vp_{j,k}$ uniformly as $\ve\to 0$, and the claim is then that
\begin{equation}\label{go3}
\tr{\omega_\ve}{\omega_{j,k,\ve}}\leq C_j,
\end{equation}
where $C_j$ does not depend on $k$ or $\ve$, and $\omega_\ve=\pi^*\omega_N+\frac{1}{j}\omega_{\ti{N}}+\ddbar\chi_{j,\ve}$ is a family of K\"ahler metrics (which also depend on $j$) that as $\ve\to 0$ approximate a ($j$-dependent) conical metric, as constructed in \cite[Section 3]{GP}. This, together with \eqref{ma3} and standard arguments, implies that $\omega_{j,k,\ve}\to\omega_{j,k}$ locally smoothly away from $E$ as $\ve\to 0$, and then \eqref{go2} follows because we allow all the constants to depend on $j$.

The proof of \eqref{go3} follows from the arguments for their Laplacian estimate in \cite[Proposition 1]{GP}, which however cannot be applied directly because they make one extra hypothesis which for us is not satisfied. The key points for us, which make the estimate independent of $k$, is that $v_k\leq C$, and $\ddbar v_k\geq -C\omega_{\ti{N}}$.

Here are the details: using the notation of \cite{GP}, the functions $\ti{\vp}=\vp_{j,k,\ve}-\chi_{j,\ve},$ are uniformly bounded (independent of $j,k,\ve$) and satisfy $\omega_\ve+\ddbar\ti{\vp}=\omega_{j,k,\ve}$. We therefore rewrite \eqref{ma3} as
$$(\omega_\ve+\ddbar\ti{\vp})^n=e^{\psi^+}\omega_\ve^n,$$
where
$$e^{\psi^+}=c_{j,k,\ve} e^{v_k-u_j}\frac{\omega_{\ti{N}}^n}{\prod_i (|s_i|^2+\ve^2)^{(1-\alpha_i)}\omega_\ve^n}=c_{j,k,\ve} e^{v_k-u_j}e^{F_\ve},$$
where $F_\ve$ is exactly defined as in \cite{GP}. Defining $\Psi_\ve$ as in \cite[Section 4]{GP}, they show that $\ddbar F_\ve\geq -(C_j\omega_\ve+\ddbar\Psi_\ve)$, and so we get that
\[\begin{split}
\ddbar\psi^+&\geq  -(C_j\omega_\ve+\ddbar\Psi_\ve)+\ddbar(v_k-u_j)\\
&\geq -(C_j\omega_\ve+\ddbar\Psi_\ve)-C_j\omega_{\ti{N}}\\
&\geq -(C_j\omega_\ve+\ddbar\Psi_\ve),
\end{split}\]
using that $\omega_\ve\geq C_j^{-1}\omega_{\ti{N}}$. This is half of inequality $(iii)$ in \cite[Proposition 1]{GP}. The second half requires that $|\psi^+|\leq C_j,$ but this is not true since $v_k$ decreases to $\log H$ which is $-\infty$ along some components of $F$, so we only have that $\psi^+\leq C_j.$ However, this is enough for their arguments, as the lower bound for $\psi^+$ is not used. Therefore, the main calculation of \cite[Proposition 1]{GP} gives
\[\begin{split}
\Delta_{\omega_{j,k,\ve}}\left(\log\tr{\omega_\ve}{\omega_{j,k,\ve}}-A_j\ti{\vp}+3\Psi_\ve\right)&\geq \tr{\omega_{j,k,\ve}}{\omega_\ve}-C_j,
\end{split}\]
and at a maximum point we get $\tr{\omega_{j,k,\ve}}{\omega_\ve}\leq C_j$ and so
$$\tr{\omega_{\ve}}{\omega_{j,k,\ve}}\leq C_j\frac{\omega_{j,k,\ve}^n}{\omega_\ve^n}=C_je^{\psi^+}\leq C_j,$$
as required.
\end{proof}

Thanks to Proposition \ref{p1}, we can apply the maximum principle to quantities involving $\tr{\omega_{\rm cone}}{\omega_{j}}$, provided we add a term which goes to $-\infty$ along $E$ (arbitrarily slowly). Using this, we first prove the main result, which uses Tsuji's trick \cite{Ts}, and from which we will quickly derive Theorem \ref{prin1}:

\begin{theorem}\label{tsuj}
On $\ti{N}\backslash E$ we have
\begin{equation}\label{goal2}
\tr{\omega_{\rm cone}}{\pi^*\omega}\leq \frac{C\psi}{|s_F|^{2A}},
\end{equation}
where $F$ is the union of the $\pi$-exceptional divisors, and $|s_F|^2$ is a shorthand for $\prod_i |s_{F_i}|_{h_i}^2$.
\end{theorem}
\begin{proof}
It is enough to show that on $\ti{N}\backslash E$ we have
\begin{equation}\label{g1}
\tr{\omega_{\rm cone}}{\omega_j}\leq \frac{Ce^{-u_j}}{|s_F|^{2A}},
\end{equation}
with $C$ now independent of $j$, since again this together with \eqref{ma2} and standard arguments implies that $\omega_j\to\pi^*\omega$ locally smoothly away from $E$.

Following \cite{GP} we define $\Psi=C\sum_i |s_i|_{h_i}^{2\rho},$ for some small $\rho>0$ and large $C>0$, which can be chosen so that on $\ti{N}\backslash E$ the curvature of $\omega_{\rm cone}$ satisfies
\begin{equation}\label{curv}
{\rm Rm}(\omega_{\rm cone})\geq -(C\omega_{\rm cone}+\ddbar\Psi)\otimes \mathrm{Id},
\end{equation}
see \cite[(4.3)]{GP}.

To prove \eqref{g1} we shall apply the maximum principle to the quantity
$$Q=\log\tr{\omega_{\rm cone}}{\omega_j}+n\Psi+u_j-A^2\vp_j+Ab\eta+\ve \log |s_E|^2+A\log|s_F|^2,$$
where $A$ large, $b>0$ is small and $0<\ve\leq\frac{1}{j}$. The terms
$n\Psi+u_j-A^2\vp_j+Ab\eta$ are all bounded on $\ti{N}$ (with bounds independent of $j$ except for $u_j$), while the term $\log\tr{\omega_{\rm cone}}{\omega_j}$ is bounded above on $\ti{N}\backslash E$ (depending on $j$) by Proposition \ref{p1}. Since the term $\ve \log |s_E|^2$ goes to $-\infty$ on $E$, the quantity $Q$ achieves a global maximum on $\ti{N}\backslash E$. All the following computations are at an arbitrary point of $\ti{N}\backslash E$.

The first claim is that on $\ti{N}\backslash E$ we have
\begin{equation}\label{c1}
\Delta_{\omega_j}(\log\tr{\omega_{\rm cone}}{\omega_j}+n\Psi)\geq -C\tr{\omega_j}{\omega_{\rm cone}}-\frac{\tr{\omega_{\rm cone}}{\Ric(\omega_j)}}{\tr{\omega_{\rm cone}}{\omega_j}}.
\end{equation}
Indeed by the Aubin-Yau's second order calculation we have
\[\begin{split}
&\Delta_{\omega_j}(\log\tr{\omega_{\rm cone}}{\omega_j}+n\Psi)\\
&\geq \frac{1}{\tr{\omega_{\rm cone}}{\omega_j}}\left(-\tr{\omega_{\rm cone}}{\Ric(\omega_j)}+{\rm Rm}(\omega_{\rm cone})_{p\ov{q}}^{\ \ \ k\ov{\ell}}g_j^{p\ov{q}}g_{j,k\ov{\ell}}\right)+n\Delta_{\omega_j}\Psi,
\end{split}\]
and as in \cite{GP} we choose coordinates so that at our given point $\omega_{\rm cone}$ is the identity while $\omega_j$ is diagonal with eigenvalues $\lambda_k>0$, so we can write
\[\begin{split}{\rm Rm}(\omega_{\rm cone})_{p\ov{q}}^{\ \ \ k\ov{\ell}}g_j^{p\ov{q}}g_{j,k\ov{\ell}}&=\sum_{i,k}\frac{\lambda_i}{\lambda_k}R_{i\ov{i}k\ov{k}}\\
&\geq -\sum_{i,k}\frac{\lambda_i}{\lambda_k}(C+\Psi_{k\ov{k}}),
\end{split}\]
using \eqref{curv}, while we also have $\ddbar\Psi\geq -C\omega_{\rm cone}$ by \cite[(4.2)]{GP}, so
\[\begin{split}
\Delta_{\omega_j}\Psi&=\sum_k\frac{1}{\lambda_k}(C+\Psi_{k\ov{k}})-C\tr{\omega_j}{\omega_{\rm cone}}\\
&\geq \frac{1}{n\tr{\omega_{\rm cone}}{\omega_j}}\sum_{i,k}\frac{\lambda_i}{\lambda_k}(C+\Psi_{k\ov{k}})-C\tr{\omega_j}{\omega_{\rm cone}},
\end{split}\]
and \eqref{c1} follows.

The second claim is that on $\ti{N}\backslash E$ we have
\begin{equation}\label{c2}\begin{split}
\Delta_{\omega_j}(\log\tr{\omega_{\rm cone}}{\omega_j}+n\Psi+u_j)&\geq-C\tr{\omega_j}{\omega_{\rm cone}}.\end{split}
\end{equation}
For this, recall that from \eqref{ma2}, on $\ti{N}\backslash E$ we have
\[\begin{split}
\Ric(\omega_j)&=\ddbar(u_j-\log H)+\Ric(\omega_{\ti{N}})-\sum_i(1-\alpha_i)R_i\\
&\geq-\frac{1}{j}\omega_{\ti{N}}\geq -\frac{C}{j}\omega_{\rm cone},
\end{split}\]
using \eqref{ein} and \eqref{zwei}, and also
\[\begin{split}\Delta_{\omega_j}u_j&=\tr{\omega_j}{\Ric(\omega_j)}+\tr{\omega_j}{\left(-\sum_j \beta_j R_{F_j}\right)}-\tr{\omega_j}{\Ric(\omega_{\ti{N}})}+\tr{\omega_j}{\left(\sum_i (1-\alpha_i) R_{i}\right)}\\
&\geq \tr{\omega_j}{\Ric(\omega_j)}-C\tr{\omega_j}{\omega_{\rm cone}}
\end{split}\]
so
\begin{equation}\label{t1}\begin{split}
\Delta_{\omega_j}&(\log\tr{\omega_{\rm cone}}{\omega_j}+n\Psi+u_j)\geq -C\tr{\omega_j}{\omega_{\rm cone}}-\frac{\tr{\omega_{\rm cone}}{\Ric(\omega_j)}}{\tr{\omega_{\rm cone}}{\omega_j}}
+\tr{\omega_j}{\Ric(\omega_j)}.
\end{split}\end{equation}
The subclaim is now that
\begin{equation}\label{sc2}
-\frac{\tr{\omega_{\rm cone}}{\Ric(\omega_j)}}{\tr{\omega_{\rm cone}}{\omega_j}}+\tr{\omega_j}{\Ric(\omega_j)}\geq -\frac{C}{j}\tr{\omega_j}{\omega_{\rm cone}}.
\end{equation}
Indeed,
\[\begin{split}
-\frac{\tr{\omega_{\rm cone}}{\Ric(\omega_j)}}{\tr{\omega_{\rm cone}}{\omega_j}}&+\tr{\omega_j}{\Ric(\omega_j)}=-\frac{\tr{\omega_{\rm cone}}{(\Ric(\omega_j)+\frac{C}{j}\omega_{\rm cone})}}{\tr{\omega_{\rm cone}}{\omega_j}}\\
&+\tr{\omega_j}{\left(\Ric(\omega_j)+\frac{C}{j}\omega_{\rm cone}\right)}+\frac{Cn}{j\tr{\omega_{\rm cone}}{\omega_j}}-\frac{C}{j}\tr{\omega_j}{\omega_{\rm cone}}\\
&\geq -\frac{\tr{\omega_{\rm cone}}{(\Ric(\omega_j)+\frac{C}{j}\omega_{\rm cone})}}{\tr{\omega_{\rm cone}}{\omega_j}}+\tr{\omega_j}{\left(\Ric(\omega_j)+\frac{C}{j}\omega_{\rm cone}\right)}
-\frac{C}{j}\tr{\omega_j}{\omega_{\rm cone}},
\end{split}\]
and if at any given point we choose coordinates so that $\omega_{\rm cone}$ is the identity and $\omega_j$ is diagonal with eigenvalues $\lambda_k>0$, while
$(\Ric(\omega_j)+\frac{C}{j}\omega_{\rm cone})\geq 0$ has entries $A_{i\ov{j}}$, then we note that
$$-\frac{1}{\sum_k\lambda_k}\sum_\ell A_{\ell\ov{\ell}}+\sum_\ell\lambda_\ell^{-1}A_{\ell\ov{\ell}}\geq 0,$$
which proves \eqref{sc2}. Combining \eqref{sc2} with \eqref{t1} we obtain \eqref{c2}.

Continuing our calculation, and setting for simplicity $R_F=\sum_iR_{F_i}$, we have
\[\begin{split}
\Delta_{\omega_j}Q&\geq -C\tr{\omega_j}{\omega_{\rm cone}}+A^2\tr{\omega_j}{\pi^*\omega_N}+Ab\tr{\omega_j}{\ddbar\eta}-A^2n\\
&\ \ \ -\ve\tr{\omega_j}{R_E}-A\tr{\omega_j}{R_F}\\
&\geq  -C\tr{\omega_j}{\omega_{\rm cone}}+A\tr{\omega_j}{\left(A_0\pi^*\omega_N- R_{F}+b\ddbar\eta\right)}-A^2n,
\end{split}\]
for all $A_0\leq A.$ We choose first $A_0$ such that $A_0\pi^*\omega_N-R_F$ is a K\"ahler metric $\hat{\omega}_{\ti{N}}$ (cf. \eqref{asin}), and then we choose $b$ sufficiently small so that
$\hat{\omega}_{\ti{N}}+b\ddbar\eta=\hat{\omega}_{\rm cone}$ is a conical K\"ahler metric with same cone angles as $\omega_{\rm cone}$. This implies that $\hat{\omega}_{\rm cone}\geq c\omega_{\rm cone}$ for some $c>0$, and finally we can choose $A\geq A_0$ large so that
$$\Delta_{\omega_j}Q\geq \tr{\omega_j}{\omega_{\rm cone}}-C.$$
Therefore at a maximum of $Q$ (which is not on $E$) we have
$$\tr{\omega_j}{\omega_{\rm cone}}\leq C,$$
and so using \eqref{ma2} we get
$$\tr{\omega_{\rm cone}}{\omega_j}\leq (\tr{\omega_j}{\omega_{\rm cone}})^{n-1}\frac{\omega_j^n}{\omega_{\rm cone}^n}\leq C H e^{-u_j},$$
hence
$$\log\tr{\omega_{\rm cone}}{\omega_j}+u_j\leq C\log H\leq C,$$
and so also $Q\leq C$, hence this last one holds everywhere on $\ti{N}\backslash E$. The constants do not depend on $\ve$, so we can let $\ve\to 0$ and this gives
$$\tr{\omega_{\rm cone}}{\omega_j}\leq \frac{Ce^{-u_j}}{|s_F|^{2A}},$$
which is \eqref{g1}.
\end{proof}

\begin{proof}[Proof of Theorem \ref{prin1}]
We have already proved in Proposition \ref{cod2} that $\omega$ extends to a smooth K\"ahler metric across $D^{(2)}$ (on $N^{\rm reg}$). First we work on the blown-up manifold $\ti{N}$ as above. Theorem \ref{tsuj} gives us the estimate \eqref{goal2}, which combined with \eqref{ma4} gives on $\pi^{-1}(N_0)$
\[\begin{split}
\tr{\pi^*\omega}{\omega_{{\rm cone}}}&\leq (\tr{\omega_{{\rm cone}}}{\pi^*\omega})^{n-1}\frac{\omega_{{\rm cone}}^n}{\pi^*\omega^n}\leq C \frac{\psi^{n-1}}{|s_F|^{2A'}}\frac{1}{H\psi}\\
&\leq\frac{C}{H|s_F|^{2A'}}\left(1-\sum_{i=1}^\mu\log|s_i|_{h_i}\right)^{\max(d(n-2),0)}
\end{split}\]
and so together with \eqref{goal2} we obtain on $\pi^{-1}(N_0)$
\begin{equation}\label{to0ter}
\begin{split}
&C^{-1}H|s_F|^{2A'}\left(1-\sum_{i=1}^\mu\log|s_i|_{h_i}\right)^{-\max(d(n-2),0)}\omega_{\rm cone}\\
&\leq \pi^*\omega\leq \frac{C}{|s_F|^{2A}}\left(1-\sum_{i=1}^\mu\log|s_i|_{h_i}\right)^{d}\omega_{\rm cone}.
\end{split}
\end{equation}
Now, if $x\in\ti{N}$ is any point near which $\pi$ is an isomorphism (which by construction includes the preimages of all points where $D^{(1)}$ is snc) then near $x$ we have $J_\pi>0$ (where $J_\pi$ is defined in \eqref{jac}) and $|s_F|>0$, and from \eqref{asino2} we see that we also have $H>0$, and so \eqref{to0bis} follows directly from \eqref{to0ter}.
\end{proof}
\section{Gromov-Hausdorff collapsed limits of Ricci-flat metrics}\label{sectmain}
\subsection{Proof of Corollary \ref{mainthm}}
In this section we show how Corollary \ref{mainthm} follows from the estimate \eqref{quasiisom}. The argument is essentially contained in the previous work of the second and third-named authors \cite[Theorem 2.1 and \S 3.4]{TZ2}, based on our work \cite{GTZ2}, where we proved the following:
\begin{theorem}\label{main}
Suppose that there is a constant $C>0$, rational numbers $0<\alpha_i\leq 1$, $1\leq i\leq \mu$, and $A>0$ such that on $\pi^{-1}(N_0)$ we have
\begin{equation}\label{to0}
\pi^*\omega\leq C\left(1-\sum_{i=1}^\mu\log|s_i|_{h_i}\right)^A \omega_{\rm cone},
\end{equation}
where $\omega_{\rm cone}$ is a conical metric with cone angle $2\pi\alpha_i$ along each component $E_i$.
Then parts (a) and (b) of Conjecture \ref{con1} hold. If furthermore we have that $N$ is smooth and $D^{(1)}$ is snc (so that $\pi=\mathrm{Id}$), then part (c) of Conjecture \ref{con1} also holds.
\end{theorem}
\begin{proof}
The first statement is proved in \cite[Theorem 2.1]{TZ2}, assuming that $\omega_{\rm cone}$ is an orbifold K\"ahler metric (i.e. $\alpha_i$ of the form $\frac{1}{N_i}$ for some $N_i\in\mathbb{N}$, for all $i$) and that $N$ is smooth. The smoothness of $N$ was actually never used in the proof there. As for the orbifold metric, since given any rational number $0<\alpha\leq 1$ we can find an integer $p$ such that $1-\alpha<1-\frac{1}{p}<1$, we can bound $\omega_{\rm cone}\leq \omega_{\rm orb}$ (where $\omega_{\rm orb}$ has orbifold order $p$ along all $E_i$), so parts (a) and (b) of Conjecture \ref{con1} follow immediately from \cite[Theorem 2.1]{TZ2}.

Next, assume that $N$ is smooth and $D^{(1)}$ is an snc divisor in $N$, and let $(X,d_X)$ the metric completion of $(N_0,\omega)$, which we want to show is homeomorphic to $N$. The argument here is contained in \cite[\S 3.4]{TZ2}, and we briefly recall it. Thanks to \eqref{to0} and the arguments in \cite[\S 3.4]{TZ2}, we know that there are continuous maps
$$h:(X,d_X)\to (N,\omega_N),$$
$$p:N\to X,$$
such that $h$ is Lipschitz and surjective, and $p$ is surjective and satisfies $\mathrm{Id}=h\circ p$. Therefore $p$ is also injective and hence a homeomorphism.
\end{proof}
To conclude the proof of Corollary \ref{mainthm} it suffices to note that \eqref{to0} follows immediately from \eqref{quasiisom} in the case when $N$ is smooth and $D^{(1)}$ is snc (so $\pi=\mathrm{Id}$), and therefore Theorem \ref{main} applies.\\

Of course, when $N$ is singular or $D^{(1)}$ fails to be snc, the map $\pi$ will be nontrivial, and the estimate \eqref{goal2} that we proved in the previous section is weaker than \eqref{to0} and not sufficient by itself to carry out the arguments in \cite[Theorem 2.1]{TZ2} to prove part (a) of Conjecture \ref{con1}.

\subsection{Geometry of the K\"ahler cone}
We conclude this section with a remark that connects our setting where the Ricci-flat K\"ahler metrics $\ti{\omega}_t$ which are cohomologous to $f^*\omega_N+e^{-t}\omega_M$ and collapse as $t\to \infty$, to the geometry of the K\"ahler cone $\mathcal{K}\subset H^{1,1}(M,\mathbb{R})$ of $M$ as investigated in \cite{Hu,Mag,Wi}. This remark is certainly known to the experts, see e.g., the
discussion in the introduction of \cite{Wi}. However, this calculation
does not appear to have appeared explicitly.

More precisely, let $\mathcal{K}_1=\{\alpha\in \mathcal{K}| \alpha^m=1 \}$ be the space of unit-volume K\"ahler classes on a Calabi-Yau manifold $M$, where we use here the shorthand notation $\alpha^m=\int_M\alpha^m$. This is a smooth manifold of dimension $h^{1,1}(M)-1$, and the tangent space at $\alpha\in\mathcal{K}_1$ is the primitive $(1,1)$-cohomology group $P^{1,1}(\alpha)=\{\beta \in H^{1,1}(M, \mathbb{R})| \alpha^{m-1}\cdot \beta=0\}$. There is a natural Riemannian metric $G$ on $\mathcal{K}_1$ introduced by Huybrechts \cite{Hu} which is given by $$G(v_1,v_2)=-\alpha^{m-2}\cdot v_1 \cdot v_2,$$
for $v_1, v_2 \in P^{1,1}(\alpha)$.

Wilson pointed  out in \cite{Wi}  that degenerations of K\"{a}hler classes at finite distance with respect to $G$ correspond to Calabi-Yau varieties, and other cases are degenerations at infinite distance.
Therefore our current case is a degeneration of K\"{a}hler classes at infinite distance, which is known to experts. The completeness of the metric $G$ was also studied  in \cite[Proposition 4.4]{Mag}.

The following elementary result is formally analogous to the characterization of degenerations of complex structures on Calabi-Yau manifolds with finite Weil-Petersson distance \cite{Wa,To5,Ta0,ZhT} by characterizing our collapsing cohomology classes $[f^*\omega_N+e^{-t}\omega_M]\in\mathcal{K}$ as those which (after rescaling to volume $1$) give paths with infinite $G$-length.  For readers' convenience, we present the detailed calculations here, and  change the parametrization of our path to $s=e^{-t}$.

\begin{proposition}\label{remark}
Let $M^n$ be a compact Calabi-Yau manifold, with classes $\alpha\in\mathcal{K}$ and $\alpha_0\in\de\mathcal{K}$. For $s\in (0,1]$ let $\alpha_s=\alpha_0+s \alpha\in \mathcal{K}$, let $V_s=\alpha_s^m$, and let $\ti{\omega}_s\in \alpha_s$ be the Ricci-flat K\"{a}hler metric given by \cite{Ya}. Then the following statements are equivalent:
\begin{itemize}
\item[(a)] $\alpha_0^m=0$.
\item[(b)]  $(M,\ti{\omega}_s)$ collapses as $s\to 0$, i.e. for any metric 1-ball $B_{\ti{\omega}_s}(p,1)$, $${\rm Vol}_{\ti{\omega}_s}(B_{\ti{\omega}_s}(p,1)) \rightarrow 0.$$
\item[(c)] If we denote $\tilde{\alpha}_s=V_s^{-\frac{1}{m}}\alpha_s \in \mathcal{K}_1$, $s\in (0,1]$, then the length of the path $\{\tilde{\alpha}_s\}_{s\in (0,1]}$ with respect to $G$ is infinite.
    \item[(d)] When $s\rightarrow 0$, $\tilde{\alpha}_s$ diverges in $H^{1,1}(M, \mathbb{R})$.
\end{itemize}
\end{proposition}
\begin{proof} (a) $\Rightarrow$ (b) since $$0=\alpha_0^m=\lim\limits_{s\rightarrow 0} \alpha_s^m= \lim\limits_{s\rightarrow 0} {\rm Vol}_{\ti{\omega}_s}(M).$$
(b) $\Rightarrow $ (a):
 By \cite{To0, RuZ}, we have $${\rm diam}(M, \ti{\omega}_s)\leq C,$$ for some constant $C>0$.  The Bishop-Gromov volume comparison theorem shows
  $$\alpha_0^m= \lim\limits_{s\rightarrow 0} {\rm Vol}_{\ti{\omega}_s}(M)\leq C^{2m}\lim\limits_{s\rightarrow 0} {\rm Vol}_{\ti{\omega}_s}(B_{\ti{\omega}_s}(p,1)) =0.$$

   (a) $ \Rightarrow $ (c): Let $0\leq d<m$ be the numerical dimension of $\alpha_0$, which is characterized as the maximum $\ell\geq 1$ such that $\alpha_0^\ell\cdot\alpha^{m-\ell}>0$, and denote by $k=m-d\geq 1$.     We have $$V_s=(\alpha_0+s\alpha)^m=  s^k c_k^m \alpha^k\cdot \alpha_0^{m-k}+O(s^{k+1})$$ where $c_k^m=\frac{m!}{k!(m-k)!}$,  and the velocity
   $$v_s= \dot{\tilde{\alpha}}_s=V_s^{-\frac{1}{m}}\alpha-\frac{1}{m}V_s^{-1}\tilde{\alpha}_s\dot{V}_s.$$ Since $v_s \cdot  \tilde{\alpha}_s^{m-1}=0$, we have
   \[\begin{split}    G(v_s,v_s)&=-V_s^{-\frac{1}{m}}\alpha \cdot v_s \cdot \tilde{\alpha}_s^{m-2} \\ & = -V_s^{-\frac{2}{m}}\alpha^2 \cdot \tilde{\alpha}_s^{m-2} + \frac{\dot{V}_s}{mV_s^{1+\frac{1}{m}}}\alpha  \cdot \tilde{\alpha}_s^{m-1}\\ &=  -V_s^{-1}\alpha^2 \cdot (\alpha_0+s\alpha)^{m-2} + \frac{\dot{V}_s}{mV_s^{2}}\alpha  \cdot (\alpha_0+s\alpha)^{m-1}.
\end{split}\] We then calculate $ \dot{V_s}=k s^{k-1} c_k^m \alpha^k\cdot \alpha_0^{m-k}+O(s^{k}),$
$$ \alpha^2 \cdot (\alpha_0+s\alpha)^{m-2}=\begin{cases}
s^{k-2}c^{m-2}_{k-2}  \alpha^k\cdot \alpha_0^{m-k} +O(s^{k-1}), & {\rm if}\;k \geq 2,\\
O(1) & {\rm if}\;k =1.
\end{cases}  $$ $$\alpha \cdot (\alpha_0+s\alpha)^{m-1}=s^{k-1}c^{m-1}_{k-1}  \alpha^k\cdot \alpha_0^{m-k} +O(s^{k}). $$   Thus $$ G(v_s,v_s)=\frac{1}{s^2}\left(\frac{k(m-k)}{m^2(m-1)}+O(s)\right).$$ We obtain that $$ \int_\varepsilon^1\sqrt{G(v_s,v_s)}ds \geq - c'\log \varepsilon \rightarrow \infty,$$ as $\varepsilon \rightarrow 0$, and so the length of the path $\{\tilde{\alpha}_s\}_{s\in (0,1]}$ with respect to $G$ is infinite.

   (c) $\Rightarrow$ (a): Assume that $\alpha_0^m>0$. A similar calculation as above shows that $$G(v_s,v_s)\leq C,$$
     and so the length of the path $\{\tilde{\alpha}_s\}_{s\in (0,1]}$ with respect to $G$ is finite, a contradiction.

   Finally, (a) $\Leftrightarrow$ (d) since $\tilde{\alpha}_s$ diverges if and only if $V_s \rightarrow0$.
\end{proof}

\section{The K\"ahler-Ricci flow setting}\label{sectkrf}
The setup in this section will be the K\"ahler-Ricci flow setting described in the Introduction, namely we let $M$ be an $m$-dimensional compact K\"ahler manifold with $K_M$ semiample and $0<\kappa(M)=:n<m$, and $f:M\to \mathbb{P}^r$ is the semiample fiber space given by the linear system $|\ell K_M|$ for some $\ell\geq 1$ sufficiently divisible, with image $N\subset\mathbb{P}^r$ a normal projective variety of dimension $n$ (the canonical model of $M$). The map $f$ has connected fibers, and if as before we let $D\subset N$ be the singular locus of $N$ together with the critical values of $f$ on $N^{\rm reg}$, then every fiber of $f$ over $N_0:=N\backslash D$ is a Calabi-Yau $(m-n)$-fold. We write $\omega_N=\frac{1}{\ell}\omega_{FS}|_N$, so that $f^*\omega_N$ belongs to $c_1(K_M)$, and we fix also a basis $\{s_i\}$ of $H^0(M,\ell K_M)$, which defines the map $f$, and let
$$\mathcal{M}=\left((-1)^{\frac{\ell m^2}{2}}\sum_i s_i\wedge\ov{s_i}\right)^{\frac{1}{\ell}},$$
which is a smooth positive volume form on $M$ which satisfies
$$\Ric(\mathcal{M})=-\ddbar\log\mathcal{M}=-f^*\omega_M.$$
On $N_0$ we then have a twisted K\"ahler-Einstein metric $\omega=\omega_N+\ddbar\vp$ which satisfies \eqref{twisted} with $\lambda=-1$ and is constructed in \cite{ST,To1,EGZ2,CGZ} by solving the Monge-Amp\`ere equation
$$(\omega_N+\ddbar\vp)^n = e^\vp f_*\mathcal{M},$$
with $\vp\in C^0(N)$, $\vp$ is $\omega_N$-psh and smooth on $N_0$.

By \cite{ST0,ST,TWY} we know that if $\omega_0$ is any K\"ahler metric on $M$ and $\omega(t)$ is its evolution by the K\"ahler-Ricci flow
$$\frac{\de}{\de t}\omega(t)=-\Ric(\omega(t))-\omega(t),\quad \omega(0)=\omega_0,$$
then $\omega(t)$ exists for all $t\geq 0$ and as $t\to\infty$ we have that $\omega(t)\to f^*\omega$ in $C^0_{\rm loc}(M\backslash f^{-1}(D))$, so $\omega$ is the collapsed limit of the K\"ahler-Ricci flow, away from the singular fibers of $f$.

The main result of this section is Theorem \ref{prin2}. The first ingredient is the following generalization of the estimates in section \ref{sectvol} in Theorem \ref{main2}. Let $\pi:\ti{N}\to N$ be a birational morphism with $\ti{N}$ smooth and $E=\pi^{-1}(D)$ is a simple normal crossings divisor $E=\bigcup_{j=1}^\mu E_j$, and fix defining sections $\tau_j$ and smooth metrics $h_j$ for $\mathcal{O}(E_j)$.
\begin{theorem}\label{main4}
There is a constant $C>0$ and natural numbers $d\in\mathbb{N}, 0\leq p\leq \mu$ and rational numbers $\beta_i>0, 1\leq i\leq p,$ and $0<\alpha_i\leq 1$, $p+1\leq i\leq \mu$, such that on $\pi^{-1}(N_0)$ we have
\begin{equation}\label{ggg}
C^{-1}\prod_{j=1}^p|\tau_j|_{h_j}^{2\beta_j}\omega_{\rm cone}^n\leq \pi^*f_*\mathcal{M}\leq C\prod_{j=1}^p|\tau_j|_{h_j}^{2\beta_j}\left(1-\sum_{i=1}^\mu\log|\tau_i|_{h_i}\right)^d \omega_{\rm cone}^n,
\end{equation}
where $\omega_{\rm cone}$ is a conical metric with cone angle $2\pi\alpha_i$ along the components $E_i$ with $p+1\leq i\leq \mu$.
\end{theorem}
\begin{proof}
As in the proof of Theorem \ref{main2}, it suffices to prove the analogous local statement (as in \eqref{estimate1}) near an arbitrary point $y_0\in E$.
We will use a ramified $\ell$-cyclic covering trick to reduce ourselves directly to Theorem \ref{volest}. Recall that $\ell K_M=f^*\mathcal{O}_{\mathbb{P}^r}(1)$ is semiample, and so we may choose the basis of sections $s_i\in H^0(M,\ell K_M)$ (which are pullbacks of linear forms on $\mathbb{P}^r$) such that their zero loci $\{s_i=0\}$ are smooth, reduced and irreducible effective divisors. At least one of these linear forms doesn't vanish at $\pi(y_0)$, and we may assume it is the one that pulls back to $s_1$ on $M$. We can then construct $\nu:\hat{M}\to M$ an $\ell$-cyclic covering ramified along $\{s_1=0\}$, see e.g. \cite[Proposition 4.1.6]{Laz}. It satisfies that $\hat{M}$ is connected and smooth, with a $\mathbb{Z}_\ell$-action with quotient $M$, and there is $\rho\in H^0(\hat{M},K_{\hat{M}})$ such that $\nu^*s_1=\rho^{\otimes\ell}$. Note that by construction $\rho$ doesn't vanish on an open set of the form $\hat{U}=\nu^{-1}(f^{-1}(U))$ for some open neighborhood $U$ of $\pi(y_0)$ in $N$. On $\hat{U}$ we may then write $\nu^*s_j=f_j \rho^{\otimes\ell}, j\geq 2$, for some holomorphic functions $f_j$. Then
$$\nu^*\mathcal{M}=\left((-1)^{\frac{\ell m^2}{2}}\sum_i \nu^*s_i\wedge\ov{\nu^*s_i}\right)^{\frac{1}{\ell}},$$
which is uniformly equivalent to
$$\sum_i \left((-1)^{\frac{\ell m^2}{2}}\nu^*s_i\wedge\ov{\nu^*s_i}\right)^{\frac{1}{\ell}}=(-1)^{\frac{m^2}{2}}\rho\wedge\ov{\rho}\left(1+\sum_{i\geq 2}|f_j|^2\right)^{\frac{1}{\ell}},$$
and so also uniformly equivalent to $(-1)^{\frac{m^2}{2}}\rho\wedge\ov{\rho}$. Thus, the aymptotic behavior of $\pi^*f_*\mathcal{M}$ is, up to constants, the same as the behavior of
$$(-1)^{\frac{m^2}{2}}\pi^*f_*\nu_*(\rho\wedge\ov{\rho}).$$
This is now almost the same setting as in Theorem \ref{volest}, the only difference is that the fibers of $f\circ\nu$ need not be connected anymore. For any $y\in U$, write $(f\circ\nu)^{-1}(y)=F_1\cup\dots\cup F_d$, where the $F_j$ are the connected components of the fiber. Then $d|\ell$ and every $F_j$ is invariant under the induced $\mathbb{Z}_{\ell/d}$-action, with $F_j/\mathbb{Z}_{\ell/d}=f^{-1}(y)$, while the quotient $\mathbb{Z}_d$-action on $(f\circ\nu)^{-1}(y)$ interchanges the components. The number $d$ is locally constant for $y\in U\backslash D$, and so we conclude by repeating the same argument as in Theorem \ref{volest} to any of the components $F_j$.
\end{proof}

\begin{proof}[Proof of Theorem \ref{prin2}]
As recalled earlier, the twisted K\"ahler-Einstein metric $\omega=\omega_N+\ddbar\vp$ on $N_0$ satisfies
$$\omega^n = e^\vp f_*\mathcal{M}=e^\vp\mathcal{F}\omega_N^n,$$
where $\vp\in C^0(N)$, $\mathcal{F}$ is defined by the last equality, and as explained above it is completely analogous to the function $\mathcal{F}$ in \eqref{deff}. It satisfies
$$\ddbar\log\mathcal{F}=-\omega_{\rm WP}+\mathrm{Ric}(\omega_N)+\omega_N.$$
As in sections \ref{codim2} and \ref{sectsnc} we see that $\mathcal{F}$ extends to a smooth positive function across $D^{(2)}\cap N^{\rm reg}$, and it satisfies $\mathcal{F}\geq C^{-1}$ since the proof of Lemma \ref{moron} goes through with minimal changes once we notice (\cite[Lemma 3.3]{ST}) that on $M\backslash S$ we have
$$f^*\mathcal{F}=\frac{\left((-1)^{\frac{\ell m^2}{2}}\sum_i s_i\wedge\ov{s_i}\right)^{\frac{1}{\ell}}}{f^*\omega_N^n\wedge\omega_{\rm SRF}^{m-n}},$$
and again the numerator is strictly positive on $M$.
As in section \ref{sectsnc}, the lower bound for $\mathcal{F}$ implies that if $N$ is smooth and $D^{(1)}$ is snc then we actually have on $N_0$
$$C^{-1}\omega_{\rm cone}^n\leq \omega^n\leq C\left(1-\sum_{i=1}^\mu\log|\tau_i|_{h_i}\right)^d \omega_{\rm cone}^n,$$
while in general we have on $\ti{N}\backslash E$
$$C^{-1}H\omega_{\rm cone}^n\leq (\pi^*\omega)^n\leq CH\left(1-\sum_{i=1}^\mu\log|\tau_i|_{h_i}\right)^d \omega_{\rm cone}^n,
$$
where $H=\prod_{j=1}^p |\tau_j|_{h_j}^{2\beta_j}$.
The proof of \eqref{quasiisom} given in section \ref{sectsnc} then goes through with the appropriate small changes, defining now the function $\psi$ by
$$\omega^n=\frac{e^\vp\psi}{\prod_{j=1}^\mu|\tau_j|^{2(1-\alpha_j)}_{h_j}}\omega_N^n,$$
and similarly, the proof of Theorem \ref{prin1} given in section \ref{sectgen} goes through, defining now $\psi$ by
$$\pi^*\omega^n=\frac{e^{\pi^*\vp}\psi H}{\prod_{j=1}^\mu|\tau_j|^{2(1-\alpha_j)}_{h_j}}\omega_{\ti{N}}^n.$$
In this way, we obtain the proof of Theorem \ref{prin2}.
\end{proof}

\end{document}